\pdfoutput=1
\RequirePackage{ifpdf}
\ifpdf 
\documentclass[pdftex]{sigma}
\else
\documentclass{sigma}
\fi

\numberwithin{equation}{section}

\newtheorem{Theorem}{Theorem}[section]

\newcommand{\Z}{\mathbb{Z}}
\newcommand{\N}{\mathbb{N}}

\DeclareMathOperator{\sn}{sn}
\DeclareMathOperator{\cn}{cn}
\DeclareMathOperator{\dn}{dn}
\DeclareMathOperator{\Ec}{Ec}
\DeclareMathOperator{\Es}{Es}
\DeclareMathOperator{\ce}{ce}
\DeclareMathOperator{\se}{se}

\DeclareMathOperator{\arctanh}{arctanh}
\DeclareMathOperator{\arcsn}{arcsn}
\def\ifrac#1#2{\textstyle{{#1}\over{#2}}\displaystyle}

\begin{document}

\allowdisplaybreaks

\renewcommand{\thefootnote}{$\star$}

\newcommand{\arXivNumber}{1507.04984}

\renewcommand{\PaperNumber}{095}

\FirstPageHeading

\ShortArticleName{Asymptotics for the Lam\'{e} and Mathieu Functions}

\ArticleName{Rigorous Asymptotics for the Lam\'{e} and Mathieu\\ Functions and their Respective Eigenvalues\\ with a Large Parameter\footnote{This paper is a~contribution to the Special Issue
on Orthogonal Polynomials, Special Functions and Applications.
The full collection is available at \href{http://www.emis.de/journals/SIGMA/OPSFA2015.html}{http://www.emis.de/journals/SIGMA/OPSFA2015.html}}}

\Author{Karen {OGILVIE} and Adri B.~{OLDE DAALHUIS}}
\AuthorNameForHeading{K.~Ogilvie and A.B.~Olde Daalhuis}
\Address{Maxwell Institute and School of Mathematics, The University of Edinburgh,\\
 Peter Guthrie Tait Road, Edinburgh EH9 3FD, UK}
\Email{\href{mailto:K.Ogilvie@sms.ed.ac.uk}{K.Ogilvie@sms.ed.ac.uk}, \href{mailto:A.OldeDaalhuis@ed.ac.uk}{A.OldeDaalhuis@ed.ac.uk}}

\ArticleDates{Received August 03, 2015, in f\/inal form November 20, 2015; Published online November 24, 2015}

\Abstract{By application of the theory for second-order linear dif\/ferential equations with two turning points developed in~[Olver F.W.J., \textit{Philos. Trans. Roy. Soc. London Ser.~A} \textbf{278} (1975), 137--174],
uniform asymptotic approximations are obtained in the f\/irst part of this paper for the Lam\'e and Mathieu functions
with a large real parameter.
These approximations are expressed in terms of parabolic cylinder functions, and are uniformly valid in their respective
real open intervals. In all cases explicit bounds are supplied for the error terms associated with the approximations.
Approximations are also obtained for the large order behaviour for the respective eigenvalues.
We restrict ourselves to a two term uniform approximation.
Theoretically more terms in these approximations could be computed, but the coef\/f\/icients would be very complicated.
In the second part of this paper we use a simplif\/ied method to obtain uniform asymptotic expansions for these functions.
The coef\/f\/icients are just polynomials and  satisfy simple recurrence relations.
The price to pay is that these asymptotic expansions hold only in a shrinking interval as their respective parameters
become large; this interval however encapsulates all the interesting oscillatory behaviour of the functions.
This simplif\/ied method also gives many terms in asymptotic expansions for these eigenvalues, derived simultaneously
with the coef\/f\/icients in the function expansions. We provide rigorous realistic error bounds
for the function expansions when truncated and order estimates for the error when the eigenvalue expansions are truncated.
With this paper we conf\/irm that many of the formal results in the literature are correct.}

\Keywords{Lam\'{e} functions; Mathieu functions; uniform asymptotic approximations; coalescing turning points}

\Classification{33E10; 34E05; 34E20}

\renewcommand{\thefootnote}{\arabic{footnote}}
\setcounter{footnote}{0}

\section{Introduction}

The Lam\'{e} equation is
\begin{gather}\label{LDiff1}
\frac{d^2w}{dz^2}+\left(h-\nu(\nu+1)k^2\sn^2(z,k)\right)w=0,
\end{gather}
where $h$, $k$ and $\nu$ are real parameters such that $0<k<1$ and $\nu\geq -\frac{1}{2}$, and $\sn(z,k)$ is a Jacobian elliptic function (see \cite[\S~22.2]{NIST:DLMF}).
We consider the interval $z\in(-K,K)$, where $K=K(k)$ is Legendre's complete elliptic integral of the f\/irst kind (see \cite[\S~19.2(ii)]{NIST:DLMF}).
When $h$ assumes the special values $a_{\nu}^{m}$ or $b_{\nu}^{m+1}$ for $m=0,1,\dots $, Lam\'{e}'s equation admits even or odd periodic solutions denoted
$\Ec_{\nu}^{m} (z,k^2 )$ or $\Es_{\nu}^{m+1} (z,k^2 )$ respectively.

Lam\'{e}'s equation f\/irst appeared in a paper by Gabriel Lam\'{e} in 1837~\cite{Lame1837}. It appears in the me\-thod of separation of variables applied to the Laplace equation in elliptic coordinates.
Lam\'{e} functions have applications in antenna research, occur when studying bifurcations in chaotic Hamiltonian systems, and in the theory of Bose--Einstein condensates, to name a few (see \cite[\S~29.19]{NIST:DLMF}).

Mathieu's equation is
\begin{gather}\label{MDiff1}
\frac{d^2w}{dz^2}+\left(\lambda-2h^2\cos2z\right)w=0,
\end{gather}
where $\lambda$ and $h$ are real parameters. We consider the interval $z\in(0,\pi)$. When $\lambda$ assumes the special values $a_{m}$ or $b_{m+1}$ for $m=0,1,2,\dots$,
Mathieu's equation admits even or odd periodic solutions denoted $\ce_{m} (h,z )$ or $\se_{m+1}(h,z)$ respectively.

These functions f\/irst arose in physical applications in 1868 in \'{E}mile Mathieu's study of vib\-rations in an elliptic drum~\cite{Mathieu1868}.
Since they have appeared in problems pertaining to vibrational systems, electrical and thermal dif\/fusion, electromagnetic wave guides, elliptical cylinders in viscous f\/luids,
and dif\/fraction of sound and electromagnetic waves, to name a few. In general, they appear when studying solutions of dif\/ferential equations that are separable in elliptic cylindrical coordinates.
For an insight to see how Mathieu functions appear in physical applications see~\cite{McLachlan1964}.

We wish to obtain uniform asymptotic approximations for the Lam\'{e} and Mathieu functions, and asymptotic expansions for their respective eigenvalues, as the parameters~$\nu$ in
Lam\'{e}'s equation and $h$ in Mathieu's equation become large. We denote for the moment the parameter $h$ in Lam\'{e}'s equation to be $h_{L}$, to avoid
confusion with the parameter~$h$ in Mathieu's equation. We have in the limit $k\to0_{+}$ from \cite[\S~22.5(ii)]{NIST:DLMF} that
\begin{gather*}
\lim_{k\to0_{+}}\sn(z,k)=\sin{z}, \qquad \textrm{and} \qquad \lim_{k\to0_{+}} K(k)=\frac{\pi}{2},
\end{gather*}
thus if $\nu\to\infty$ in such a way that as $k\to0_{+}$, $\sqrt{\nu(\nu+1)}k=2h$ for some constant $h$, then we can rewrite the limit of Lam\'{e}'s equation in the form
\begin{gather*}
\frac{d^2w}{dz^2}+\left(h_{L}-2h^2-2h^2\cos{2\left(\frac{\pi}{2}-z\right)}\right)w=0.
\end{gather*}
Thus for $\nu=-\frac{1}{2}+\sqrt{\frac{1}{4}+\left(\frac{2h}{k}\right)^2}$ we have
\begin{gather}\label{Msoln}
\lim_{k\to0_{+}} \Ec_{\nu}^{m}\left(z,k^2\right)=\ce_{m}\left(h,\frac{\pi}{2}-z\right),\qquad
\lim_{k\to0_{+}}  \Es_{\nu}^{m+1}\left(z,k^2\right)=\se_{m+1}\left(h,\frac{\pi}{2}-z\right),
\end{gather}
and
\begin{align}\label{MEigen1}
a_{m}=\lim_{k\to0_{+}}a_{\nu}^{m}-2h^2,\qquad
b_{m+1}=\lim_{k\to0_{+}}b_{\nu}^{m+1}-2h^2.
\end{align}
With this in mind, we derive rigorous results for the Lam\'{e} functions and their eigenvalues, and deduce analogous results for Mathieu's equation using this limiting relation.

For a general overview of the Lam\'{e} and Mathieu equations, see \cite{Arscott64,HTF3,WW1927}. For a more detailed study of Mathieu's equation see also \cite{MS1954}.
Whilst the results for asymptotic expansions of the Lam\'{e} functions and their respective eigenvalues for parameter~$\nu$ large are not so abundant, analogous problems for the Mathieu functions and their eigenvalues for parameter $h$ large have been studied extensively. The main results in the Lam\'{e} case can be found in \cite{Ince1940b,Ince1940a, Muller1966a, Muller1966b,Muller1966c}. These results are all formal, meaning they are not accompanied with any error analysis. None of the results about the functions have been published in~\cite[\S~29.7]{NIST:DLMF}, and only limited formal results about the corresponding eigenvalues can be found there. In~\cite[\S~29.7(ii)]{NIST:DLMF} it is stated that one could derive from the results of~\cite{Weinstein1985} asymptotic approximations for the Lam\'{e} functions. However in that paper the results are given without much justif\/ication and the error bounds given for the approximations do not make sense in the intervals where the approximant is exponentially small. Here the results we give for the Lam\'{e} functions are new; we give a two term uniform asymptotic approximation for the Lam\'{e} functions in terms of parabolic cylinder functions for~$\nu$ large complete with error bounds, and we show this holds uniformly in the interval $z\in [0,K]$. We also make a rigorous statement about the corresponding eigenvalues.
We also give asymptotic expansions for the functions and eigenvalues in a shrinking neighbourhood of the origin, which correspond with the few formal results in the literature.

The main results in the Mathieu case can be found in \cite{Dingle1962, Dunster1994,Goldstein1927, Goldstein1929,Ince1927a, Ince1927b, Ince1928, Kurz1979, MS1954,Sips1949, Sips1959, Sips1965,Weinstein1985}. In~\cite{MS1954,Weinstein1985}, error estimates are written down for a one term asymptotic approximation of the Mathieu functions but these are given without any justif\/ication, and the error does not make sense in the intervals where the approximant is exponentially small. In~\cite{Kurz1979}, similar results are given for the functions
with similar issues for the error estimates; the results about the eigenvalues however seem reasonable, but methods of obtaining terms in the eigenvalue expansions seem cumbersome.
The most satisfactory work thus far is contained in~\cite{Dunster1994}. Here Dunster derives uniform asymptotic approximations for all complex values $z$ when $-2h^2\leq \lambda \leq (2-d)h^2$, $d>0$, with error bounds either included or available for all approximations. These approximations involve both elementary functions and Whittaker functions. He also includes rigorous statements related to the eigenvalues~$a_{m}$ and~$b_{m+1}$. The remaining papers stated include only formal results, without any satisfactory error analysis. Here we consider only the real interval $z\in(0,\pi)$ as many physical applications are restricted to real variables. The results we give are uniform asymptotic approximations complete with error bounds in the interval $z\in\big[0,\frac{\pi}{2}\big]$, given in the most natural form for the case we consider. Since we restrict ourselves to real variable analysis we can make stronger statements about the error bounds for the functions and their respective eigenvalues than given in~\cite{Dunster1994}.
We also give asymptotic expansions for the functions and eigenvalues in a~shrinking neighbourhood of~$\frac{\pi}{2}$, which correspond with the formal results in the literature.

\section{Overview}

This paper is written in two parts: In Part \ref{partone} we derive two term uniform asymptotic approximations for the Lam\'{e} functions $\Ec_{\nu}^{m}\left(z,k^2\right)$ and $\Es_{\nu}^{m+1}\left(z,k^2\right)$ which hold for $z\in[0,K]$, and rigorous approximations for the large order behaviour of their respective eigenvalues~$a_{\nu}^{m}$ and~$b_{\nu}^{m+1}$, $m=0,1,\dots $, as $\kappa\to\infty$, where $\kappa=\sqrt{\nu(\nu+1)}k$. Treating the Mathieu functions and their respective eigenvalues as a special case of those in the Lam\'{e} case, we obtain simply uniform asymptotic approximations for the Mathieu functions $\ce_{m}(h,z)$ and $\se_{m+1}(h,z)$ which hold for $z\in\big[0,\frac{\pi}{2}\big]$, and rigorous approximations for the large order behaviour of their respective eigenvalues $a_{m}$ and $b_{m+1}$, as $h\to\infty$.

Part~\ref{parttwo} uses a simplif\/ied method to derive asymptotic expansions for both the Lam\'{e} and Mathieu functions
and their respective eigenvalues. We can compute as many terms as we like in these expansions.
The price to pay is that these function expansions only hold for
$z=\mathcal{O}(\kappa^{-1/2})$ and $z=\frac{\pi}{2}+\mathcal{O}(h^{-1/2})$, as $\kappa\to\infty$ and $h\to\infty$ respectively. These intervals at least encapsulate all of the interesting oscillatory behaviour of the functions. We give rigorous and realistic error bounds for the function expansions once truncated, along with order estimates for the error when the eigenvalue expansions are truncated.

First, we will summarise the relevant properties of the Lam\'{e} and Mathieu functions and their respective eigenvalues in the upcoming section.

\section{Properties of Lam\'{e} and Mathieu functions}
\subsection*{Lam\'{e} functions}

We will summarise their important properties here, for a full treatment see~\cite[\S~29]{NIST:DLMF}. These functions are either $2K$-periodic or $2K$-antiperiodic, depending on the parity of~$m$. The functions have exactly~$m$ zeros in the interval $(-K,K)$ and their eigenvalues are ordered such that
\begin{gather*}
a_{\nu}^{0}<a_{\nu}^{1}<a_{\nu}^{2}<\cdots ,\qquad b_{\nu}^{1}<b_{\nu}^{2}<b_{\nu}^{3}<\cdots ,
\end{gather*}
and interlace such that
\begin{gather*}
a_{\nu}^{m}<b_{\nu}^{m+1},\qquad
b_{\nu}^{m}<a_{\nu}^{m+1}.
\end{gather*}
The eigenvalues coalesce such that
\begin{gather*}
a_{\nu}^{m}=b_{\nu}^{m}, \qquad \textrm{when} \quad \nu=0,1,\dots,m-1.
\end{gather*}
Since the Jacobian elliptic function $\dn(z,k)$ (see \cite[\S~22]{NIST:DLMF}) is even, we can rewrite the normalisations given in \cite[\S~29.3]{NIST:DLMF} as
\begin{gather}\label{LNorm}
\int_{-K}^{K}\dn(z,k)\left\{\Ec_{\nu}^{m}\left(z,k^2\right)\right\}^2dz=\int_{-K}^{K}\dn(z,k)\left\{\Es_{\nu}^{m+1}\left(z,k^2\right)\right\}^2dz=\frac{\pi}{2},
\end{gather}
To complete their def\/initions we have
\begin{gather} \label{Lsigns}
\Ec_{\nu}^{m}\left(K,k^2\right) >0,\qquad {\rm and} \qquad
\frac{d\Es_{\nu}^{m}\left(z,k^2\right)}{dz}\bigg|_{z=K} <0.
\end{gather}
They satisfy the orthogonality conditions for $m \neq n$, $(n=0,1,\dots )$
\begin{gather*}
\int_{-K}^{K}\Ec_{\nu}^{m}\left(z,k^2\right)\Ec_{\nu}^{n}\left(z,k^2\right)dz=0,\qquad {\rm and} \qquad
\int_{-K}^{K}\Es_{\nu}^{m+1}\left(z,k^2\right)\Es_{\nu}^{n+1}\left(z,k^2\right)dz=0.
\end{gather*}
We summarise their properties and give boundary conditions in Table \ref{TableA}.

\begin{table}[h!]\centering
\caption{Properties and boundary conditions for Lam\'{e} functions.}\label{TableA}
\vspace{1mm}

\begin{tabular}{c c c c c c}
\hline
 eigenfunctions &  eigenvalues&periodicity& $\begin{array}{@{}c@{}} \text{parity} \\ \text{at $z=K$}\end{array}$
 & $\begin{array}{@{}c@{}} \text{parity} \\ \text{at $z=0$}\end{array}$ & boundary conditions\\
\hline
$\Ec_{\nu}^{2m}\left(z,k^2\right)$&$a_{\nu}^{2m}$&period $2K$ & even & even &$w'(0)=w'(K)=0$ \tsep{2pt}\bsep{1pt}\\
$\Ec_{\nu}^{2m+1}\left(z,k^2\right)$&$a_{\nu}^{2m+1}$&antiperiod $2K$ &even & odd& $w(0)=w'(K)=0$ \bsep{1pt}\\
$\Es_{\nu}^{2m+1}\left(z,k^2\right)$&$b_{\nu}^{2m+1}$&antiperiod $2K$ & odd & even&$w'(0)=w(K)=0$ \bsep{1pt}\\
$\Es_{\nu}^{2m+2}\left(z,k^2\right)$&$b_{\nu}^{2m+2}$&period $2K$  &odd & odd &$w(0)=w(K)=0$\bsep{1pt}\\
\hline
\end{tabular}
\end{table}

\subsection*{Mathieu functions}

We will summarise their important properties here, for a full treatment see \cite[\S~28]{NIST:DLMF}. These are either $\pi$-periodic or $\pi$-antiperiodic, depending on the parity of $m$. Both functions have $m$ zeros in the interval $(0,\pi)$ and their eigenvalues are ordered such that
\begin{gather*}
a_{0}<a_{1}<\cdots \to \infty,\qquad
b_{1}<b_{2}<\cdots \to \infty,
\end{gather*}
and interlace such that
\begin{gather*}
a_{0}<b_{1}<a_{1}<b_{2}<a_{2}<\cdots .
\end{gather*}
The normalisations in \cite[\S~28.2]{NIST:DLMF} can be rewritten as
\begin{gather*}
\int_{0}^{\pi}\left\{\ce_{m}(h,z)\right\}^2dz=\int_{0}^{\pi}\left\{\se_{m+1}(h,z)\right\}^2dz=\frac{\pi}{2},
\end{gather*}
and to complete their def\/initions, the signs are determined by continuity from
\begin{gather*}
\ce_{0}(0,z)=\frac{1}{\sqrt{2}}, \qquad \ce_{m}(0,z)=\cos{mz}, \qquad \se_{m}(0,z)=\sin{mz}.
\end{gather*}
We summarise their properties and give boundary conditions in Table~\ref{TableB}.

\begin{table}[h]\centering
\caption{Properties and boundary conditions for Mathieu functions.}\label{TableB}
\vspace{1mm}

\begin{tabular}{c c c c c c}
\hline
 eigenfunctions &  eigenvalues&periodicity &$\begin{array}{@{}c@{}} \text{parity  of} \\ \text{$w(h,z)$}\end{array}$
   & $\begin{array}{@{}c@{}} \text{parity  of} \\ \text{$w\left(h,z+\frac{\pi}{2}\right)$}\end{array}$ & boundary conditions\\
\hline
$\ce_{2m}(h,z)$&$a_{2m}$&period $\pi$ & even& even& $w'(0)=w'\left(\pi/2\right)=0$ \tsep{2pt}\bsep{1pt}\\
$\ce_{2m+1}(h,z)$&$a_{2m+1}$&antiperiod $\pi$ &even& odd&$w'(0)=w\left(\pi/2\right)=0$ \bsep{1pt}\\
$\se_{2m+1}(h,z)$&$b_{2m+1}$&antiperiod $\pi$ &odd& even&$w(0)=w'\left(\pi/2\right)=0$ \bsep{1pt}\\
$\se_{2m+2}(h,z)$&$b_{2m+2}$&period $\pi$  &odd&  odd&$w(0)=w\left(\pi/2\right)=0$\bsep{1pt}\\
\hline
\end{tabular}
\end{table}

\part{Uniform asymptotic approximations}\label{partone}
We wish to obtain uniform approximations for the Lam\'{e} functions $\Ec_{\nu}^{m}\left(z,k^2\right)$ and $\Es_{\nu}^{m+1}\left(z,k^2\right)$, and rigorous approximations for their respective eigenvalues $a_{\nu}^{m}$ and $b_{\nu}^{m+1}$, for $m=0,1,\dots $.
In~\cite[\S~29.7]{NIST:DLMF}, the f\/irst few terms are given for formal asymptotic expansions of the eigenvalues~$a_{\nu}^{m}$ and~$b_{\nu}^{m+1}$ as $\nu\to\infty$.
These indicate that~(\ref{LDiff1}) will have two turning points, one either side of the origin, which coalesce symmetrically into the origin of the $z$-plane as $\nu\to\infty$.
In Section~\ref{CaseA} we will use this fact and derive uniform approximations for solutions in terms of the parabolic cylinder functions $U\big({-}\frac{1}{2}\kappa\sigma^2,\sqrt{2\kappa}\zeta\big)$
and $\overline{U}\big({-}\frac{1}{2}\kappa\sigma^2,\sqrt{2\kappa}\zeta\big)$ as $\nu\to\infty$, where $\kappa=\sqrt{\nu(\nu+1)}k$ and $\zeta$ arises from a special complicated transformation of the
variable~$z$, and~$\sigma$ is related to the eigenvalue parameter~$h$.
Only when $-\frac{1}{2}\kappa\sigma^2$ is exactly a negative half integer does this function decay exponentially on both the positive and negative real axis when the variable is large (see \cite[\S~12.9]{NIST:DLMF}).
Respective of this, in Section~\ref{interlude} we get a rigorous approximation for the eigenvalues~$a_{\nu}^{m}$ and~$b_{\nu}^{m+1}$ and use this in Section~\ref{CaseB} to give an approximation in
terms of the special parabolic cylinder functions $D_{m}\big(\sqrt{2\kappa}\zeta\big)$, where $D_{m}(z)=U\left(-m-\frac{1}{2},z\right)$, and these decay exponentially for both large positive and large negative
values of the variable, and have exactly~$m$ zeros in their oscillatory interval (again see \cite[\S~12.2]{NIST:DLMF}). Finally in Section~\ref{Mathieu} we use these results for the Lam\'{e} equation to obtain in
a limiting form uniform approximations for the Mathieu functions $\ce_{m}(h,z)$ and $\se_{m+1}(h,z)$, and rigorous results for their eigenvalues~$a_{m}$ and~$b_{m+1}$, as $h\to\infty$.

\section{Uniform approximations for the Lam\'{e} functions}\label{CaseA}

The periodic coef\/f\/icient in (\ref{LDiff1}) is troublesome, thus we transform the independent variable to obtain an algebraic equation, and then transform the dependent variable to remove the subsequent term in the f\/irst derivative; this is done by letting
\begin{gather}\label{LTran}
x=\sn(z,k), \qquad w\left(z,k^2\right)=\frac{1}{\left(\left(1-x^2\right)\left(1-k^2x^2\right)\right)^{1/4}}\widetilde{w}\left(x,k^2\right),
\end{gather}
and denoting $\kappa=\sqrt{\nu(\nu+1)}k$, we obtain Lam\'{e}'s equation in the form
\begin{gather}\label{LDiff2}
\frac{d^2\widetilde{w}\left(x,k^2\right)}{dx^2}=\left(\kappa^2\frac{x^2-s^2}{\left(1-x^2\right)\left(1-k^2x^2\right)}
+\phi\left(x,k^2\right)\right)\widetilde{w}\left(x,k^2\right),
\end{gather}
where
\begin{gather}\label{LEigen1}
s^2=\frac{h}{\kappa^2}
\end{gather}
and
\begin{gather*}
\phi\left(x,k^2\right)=-\frac{2k^2\left(k^2+1\right)x^4+\left(k^4-10k^2+1\right)x^2+2\left(1+k^2\right)}{4\left(1-x^2\right)^2\left(1-k^2x^2\right)^2}.
\end{gather*}
We correspondingly consider the interval $x\in (-1,1 )$, where $x=-1,0,1$ corresponds to $z=-K,0,K$. We now write $\kappa\to\infty$ to correspond to $\nu\to\infty$.

From formal asymptotic expansions given in \cite[\S~29.7]{NIST:DLMF} we deduce that $s\to0$ as $\kappa\to\infty$, hence in this limit~(\ref{LDiff2}) has two coalescing turning points.
The turning points of our equation are at $x=\pm s$ and
\begin{gather*}
\frac{x^2-s^2}{\left(1-x^2\right)\left(1-k^2x^2\right)}<0, \qquad -s<x<s,
\end{gather*}
thus we apply the theory of Case~I in~\cite{Olver1975}. In this case uniform asymptotic approximations are in terms of the parabolic cylinder functions
$U\big({-}\frac{1}{2}\kappa\sigma^2,\sqrt{2\kappa}\zeta\big)$ and $\overline{U}\big({-}\frac{1}{2}\kappa\sigma^2,\sqrt{2\kappa}\zeta\big)$. For the standard notation see~\cite[\S~12.2]{NIST:DLMF}.

Following Olver \cite{Olver1975}, new variables relating $\left\{x,\widetilde{w}\right\}$ to $\left\{\zeta,W\right\}$ are introduced by the appropriate Liouville transformation given by
\begin{gather}\label{LLiouv}
W\left(\zeta,k^2\right)=\dot{x}^{-\tfrac{1}{2}}\widetilde{w}\left(x,k^2\right), \qquad \dot{x}^{2}\frac{x^{2}-s^2}{\left(1-x^{2}\right)\left(1-k^2x^2\right)} = \zeta^{2}-\sigma^2
\end{gather}
the dot signifying dif\/ferentiation with respect to $\zeta$, where $\sigma$ is def\/ined by
\begin{gather}\label{LEigen2}
\int_{-s}^{s}\sqrt{\frac{s^2-t^{2}}{\left(1-t^{2}\right)\left(1-k^2t^2\right)}} dt  =
\int_{-\sigma}^{\sigma}\sqrt {\sigma^2-\tau^{2}}d\tau=\frac{\pi}{2}\sigma ^2.
\end{gather}
From this we denote that
\begin{gather*}
0<s<1 \quad \textrm{corresponds to} \quad 0<\sigma<\sigma_{*}, \qquad \textrm{where} \quad \sigma_{*}=2\sqrt{\frac{\arcsin(k)}{\pi k}}.
\end{gather*}
Since $\zeta=\pm \sigma$ corresponds to $x=\pm s$, integration of the second of~(\ref{LLiouv}) yields
\begin{alignat}{3} \nonumber
& \int_{x}^{-s}\sqrt{
\frac{t^{2}-s^2}{\left(1-t^{2}\right)\left(1-k^2t^2\right)}}dt  =\int_{\zeta}^{-\sigma}\sqrt{\tau^{2}-\sigma^2}d\tau, \qquad && -1<x \leq -s,& \\
\nonumber
 & \int_{-s}^{x}\sqrt
{\frac{s^2-t^{2}}{\left(1-t^{2}\right)\left(1-k^2t^2\right)}}dt  = \int_{-\sigma}^{\zeta}\sqrt{\sigma^2-\tau^{2}} d\tau, \qquad && -s\leq x \leq s,& \\ \label{LInts}
& \int_{s}^{x}\sqrt
{\frac{t^{2}-s^2}{\left(1-t^{2}\right)\left(1-k^2t^2\right)}} dt  = \int_{\sigma}^{\zeta}\sqrt{\tau^{2}-\sigma^2}d\tau, \qquad && s \leq x < 1 .&
\end{alignat}
These equations def\/ine $\zeta$ as a real analytic function of~$x$. There is a one-to-one correspondence between the variables~$x$ and~$\zeta$, where~$\zeta$ is an increasing function of $x$, and we denote $\zeta=-\zeta_{*},0$, $\zeta_{*}$ to correspond to $x=-1,0,1$. It follows that $x(\zeta,\sigma)$ is analytic both in $\zeta$ and $\sigma$ for $\zeta\in(-\zeta_{*},\zeta_{*})$ and $\sigma\in(-\sigma_{*},\sigma_{*})$. Also $\dot{x}$ is non-zero in these intervals.

Performing the substitution $t=s\tau$ in~(\ref{LEigen2}) we expand the integral and obtain
\begin{gather*}
\sigma^2=s^2+\frac{1+k^2}{8}s^4+\frac{3+2k^2+3k^4}{64}s^6+\mathcal{O}\left(s^8\right), \qquad s\to 0,
\end{gather*}
and by reversion
\begin{gather}\label{LEigen4}
s^2=\sigma^2-\frac{1+k^2}{8}\sigma^4-\frac{\left(1-k^2\right)^2}{64}\sigma^6+\mathcal{O}\left(\sigma^8\right), \qquad \sigma\to 0.
\end{gather}
In the critical case $s=\sigma=0$ we have from the third of~(\ref{LInts})
\begin{gather*}
\frac{\arctanh(k)}{k}=\int_{0}^{1}\frac{t}{\sqrt{(1-t^2)(1-k^2t^2)}}dt =\int_{0}^{\zeta_{*}}\tau d\tau=\ifrac{1}{2}\zeta_{*}^2,
\end{gather*}
which gives
\begin{gather*}
\zeta_{*}=\sqrt{\frac{2\arctanh(k)}{k}}.
\end{gather*}
Thus we deduce that as $s,\sigma\to0$
\begin{gather*}
\zeta_{*}\to\sqrt{\frac{2\arctanh(k)}{k}}.
\end{gather*}
The transformed dif\/ferential equation is now of the form
\begin{gather}\label{LDiff2.5}
\frac{d^{2}W\left(\zeta,k^2\right)}{d\zeta^{2}} =
\left(\kappa^2\left(\zeta^{2}-\sigma^2\right)+{\psi}\left(\zeta,k^2\right)\right)W\left(\zeta,k^2\right),
\end{gather}
where
\begin{gather*}
{\psi}\left(\zeta,k^2\right) =\dot{x}^2\phi\left(x,k^2\right)+\dot{x}^{1/2}\frac{d^2}{d\zeta^2}\big(\dot{x}^{-1/2}\big)\\
\hphantom{{\psi}\left(\zeta,k^2\right)}{}
 =\frac{2\sigma^2+3\zeta^2}{4\left(\sigma^2-\zeta^2\right)^2}+\frac{\sigma^2-\zeta^2}{4}\left(-k^2+\frac{1+k^2\left(1-3s^2\right)}{s^2-x^2} \right. \\
 \left. \hphantom{{\psi}\left(\zeta,k^2\right)=}{}
 +\frac{3\left(1-2\left(1+k^2\right)s^2+3k^2s^4\right)}{\left(s^2-x^2\right)^2}+\frac{5s^2\left(s^2-1\right)\left(1-k^2s^2\right)}{\left(s^2-x^2\right)^3}\right).
\end{gather*}
By construction, the apparent singularities in the above function at $\zeta = \pm \sigma$, corresponding to $x=\pm s$, cancel each other out so that ${\psi}\left(\zeta,k^2\right)$ is well-behaved there. To check this one could expand ${\psi}\left(\zeta,k^2\right)$ around this point. One could also note that in (\ref{LDiff2}), $\phi\left(x,k^2\right)$ has singularities at $x=\pm1$, where as ${\psi}\left(\zeta,k^2\right)$ does not blow up there. However $\zeta$ and therefore $\psi\left(x,k^2\right)$ will have branch point singularities there.

For our advantage in the error analysis, we write our dif\/ferential equation in the form
\begin{gather}\label{LDiff3}
\frac{d^{2}W\left(\zeta,k^2\right)}{d\zeta^{2}} =
\big(\kappa^{2}\left(\zeta^{2}-\widetilde{\sigma}^2\right)+{\widetilde{\psi}}\left(\zeta,k^2\right)\big)W\left(\zeta,k^2\right),
\end{gather}
where
\begin{gather*}
\widetilde{\psi}\left(\zeta,k^2\right)={\psi}\left(\zeta,k^2\right)-{\psi}\left(0,k^2\right)={\psi}\left(\zeta,k^2\right)+\frac{\sigma^4-s^4}{2\sigma^2 s^4},
\end{gather*}
and correspondingly
\begin{gather}\label{LEigen44}
\widetilde{\sigma}^2={\sigma}^2+\frac{\sigma^4-s^4}{2\kappa^2\sigma^2 s^4},
\end{gather}
where
\begin{gather}\label{LLEigen}
\frac{\sigma^4-s^4}{2\sigma^2 s^4}=\frac{1+k^2}{8}+\mathcal{O}\left(s^2\right), \qquad s\to 0.
\end{gather}
This gives
\begin{gather}\label{psi0}
\widetilde{\psi}\left(0,k^2\right)=0.
\end{gather}
On inspection it follows that since $s$, $\sigma$ and the variables $x$ and $\zeta$ are all bounded as $\kappa\to\infty$, and the apparent singularities at $x=\pm s$ and $\zeta = \pm \sigma$ cancel each other, we have for $\zeta\in[-\zeta_{*},\zeta_{*}]$
\begin{gather*}
\widetilde{\psi}\left(\zeta,k^2\right)=\mathcal{O}(1),
\end{gather*}
uniformly in this limit.
On applying Theorem~I of \cite[\S~6]{Olver1975} with $u=\kappa$, $\alpha=\widetilde{\sigma}$ and $\zeta_{2}=\zeta_{*}$, we obtain the following solutions of~(\ref{LDiff3})
\begin{gather}
W_{1}\left(\zeta,k^2\right) =U\left(-\ifrac{1}{2}\kappa\widetilde{\sigma}^2,\sqrt{2\kappa}\zeta\right)+\epsilon_{1}\left(\zeta,k^2\right),\nonumber
\\
W_{2}\left(\zeta,k^2\right) =\overline{U}\left(-\ifrac{1}{2}\kappa\widetilde{\sigma}^2,\sqrt{2\kappa}\zeta\right)+\epsilon_{2}\left(\zeta,k^2\right),
\label{LSoln1}
\end{gather}
valid when $\zeta\in[0,\zeta_{*})$. Choosing as a consequence of (\ref{psi0})
\begin{gather*}
\Omega(z)=z,
\end{gather*}
and thus def\/ining the variational operator as
\begin{gather*}
\mathcal{V}_{a,b}\big(\widetilde{\psi}\big) = \int_{a}^{b}\frac{\big|\widetilde{\psi}\left(t,k^2\right)\big|}{\sqrt{2\kappa}t}dt,
\end{gather*}
the bounds obtained are
\begin{gather*}
\left|\epsilon_{1}\left(\zeta,k^2\right)\right|  \leq \frac{\textbf{M}\left(-\ifrac{1}{2}\kappa\widetilde{\sigma}^2,\sqrt{2\kappa}\zeta\right)}{\textbf{E}
\left(-\ifrac{1}{2}\kappa\widetilde{\sigma}^2,\sqrt{2\kappa}\zeta\right)}
\left[\exp\left(\ifrac{1}{2}\sqrt{\pi/\kappa} l_{1}\left(-\ifrac{1}{2}\kappa\widetilde{\sigma}^2\right)\mathcal{V}_{\zeta,\zeta_{*}}\big(\widetilde{\psi}\big)\right)-1\right],\\
\frac{\left|\epsilon_{2}\left(\zeta,k^2\right)\right|}{\textbf{E}\left(-\ifrac{1}{2}\kappa\widetilde{\sigma}^2,\sqrt{2\kappa}\zeta\right)}  \leq \textbf{M}\left(-\ifrac{1}{2}\kappa\widetilde{\sigma}^2,\sqrt{2\kappa}\zeta
\right)\left[\exp\left(\ifrac{1}{2}\sqrt{\pi/\kappa} l_{1}\left(-\ifrac{1}{2}\kappa\widetilde{\sigma}^2\right)\mathcal{V}_{0,\zeta}\big(\widetilde{\psi}\big)\right)-1\right],
\end{gather*}
where
\begin{gather*}
l_{1}\left(-\ifrac{1}{2}\kappa\widetilde{\sigma}^2\right)=\sup_{z\in(0,\infty)}
\left\{\frac{\Omega(z)\textbf{M}^2\left(-\frac{1}{2}\kappa\widetilde{\sigma}^2,z\right)}
{\Gamma\left(\frac{1}{2}\left(1+\kappa\widetilde{\sigma}^2\right)\right)}\right\}.
\end{gather*}
The functions $\textbf{E}$, $\textbf{M}$ and later $\textbf{N}$ are def\/ined in  \cite[\S~5.8]{Olver1975}.
It follows from~(\ref{psi0}) and the evenness of $\widetilde{\psi}\left(\zeta,k^2\right)$ that
\begin{gather*}
\mathcal{V}_{0,\zeta}\big(\widetilde{\psi}\big)=\mathcal{O}\big(\kappa^{-1/2}\big) \qquad \textrm{and} \qquad  \mathcal{V}_{\zeta,\zeta_{*}}\big(\widetilde{\psi}\big)=\mathcal{O}\big(\kappa^{-1/2}\big) \qquad \textrm{as} \quad \kappa\to\infty.
\end{gather*}
From  \cite[\S~5.8]{Olver1975} we have
\begin{gather}\label{Olvref}
\textbf{M}^2\left(-\ifrac{1}{2}\kappa\widetilde{\sigma}^2,z\right)=\mathcal{O}\left(z^{-1}\right)\qquad \textrm{as} \quad z\to\infty,
\end{gather}
thus clearly $l_{1}\left(-\frac{1}{2}\kappa\widetilde{\sigma}^2\right)$ is bounded as $\kappa\to\infty$.
Hence
\begin{gather}\nonumber
\epsilon_{1}\left(\zeta,k^2\right)  = \frac{\textbf{M}\left(-\ifrac{1}{2}\kappa\widetilde{\sigma}^2,\sqrt{2\kappa}\zeta\right)}{\textbf{E}\left(-\ifrac{1}{2}\kappa\widetilde{\sigma}^2,\sqrt{2\kappa}\zeta\right)}\mathcal{O}\left(\kappa^{-1}\right),\\ \label{errors1}
\epsilon_{2}\left(\zeta,k^2\right)  = \textbf{E}\left(-\ifrac{1}{2}\kappa\widetilde{\sigma}^2,\sqrt{2\kappa}\zeta
\right)\textbf{M}\left(-\ifrac{1}{2}\kappa\widetilde{\sigma}^2,\sqrt{2\kappa}\zeta
\right)\mathcal{O}\left(\kappa^{-1}\right).
\end{gather}
Applying similar analysis to the above, again from Theorem~1 in~\cite[\S~6]{Olver1975}, we obtain
\begin{gather}\nonumber
\frac{\partial \epsilon_{1}\left(\zeta,k^2\right)}{\partial\zeta}  = \frac{\textbf{N}\left(-\ifrac{1}{2}\kappa\widetilde{\sigma}^2,\sqrt{2\kappa}\zeta\right)}{\textbf{E}\left(-\ifrac{1}{2}\kappa\widetilde{\sigma}^2,\sqrt{2\kappa}\zeta
\right)}\mathcal{O}\big(\kappa^{-1/2}\big),\\ \label{errors2}
\frac{\partial \epsilon_{2}\left(\zeta,k^2\right)}{\partial\zeta}  = \textbf{E}\left(-\ifrac{1}{2}\kappa\widetilde{\sigma}^2,\sqrt{2\kappa}\zeta
\right)\textbf{N}\left(-\ifrac{1}{2}\kappa\widetilde{\sigma}^2,\sqrt{2\kappa}\zeta\right)\mathcal{O}\big(\kappa^{-1/2}\big).
\end{gather}
We can extend this analysis to include the point $\zeta_{*}$ since $\zeta$ and $\psi\left(\zeta,k^2\right)$ are bounded there.

\section{An interlude: eigenvalues}\label{interlude}

Thus we obtain from (\ref{LSoln1}), (\ref{LLiouv}) and (\ref{LTran}) that for $z\in[0,K]$
\begin{gather*}
\Ec_{\nu}^{m}\left(z,k^2\right) = C_{\nu}^{m}\left(\frac{\zeta^2-\left(\sigma_{\nu}^{m}\right)^2}{x^2-\left(s_{\nu}^{m}\right)^2}\right)^{1/4}\left(W_{\nu,1}^{m}
\left(\zeta,k^2\right)+\eta_{\nu,c}^{m}W_{\nu,2}^{m}\left(\zeta,k^2\right)\right),\\
\Es_{\nu}^{m+1}\left(z,k^2\right) = S_{\nu}^{m+1}\left(\frac{\zeta^2-\left(\sigma_{\nu}^{m}\right)^2}{x^2-\left(s_{\nu}^{m}\right)^2}\right)^{1/4}\left(W_{\nu,1}^{m}
\left(\zeta,k^2\right)+\eta_{\nu,s}^{m+1}W_{\nu,2}^{m}\left(\zeta,k^2\right)\right),
\end{gather*}
where $C_{\nu}^{m}$, $S_{\nu}^{m+1}$, $\eta_{\nu,c}^{m}$ and $\eta_{\nu,s}^{m+1}$ are constants, and $s_{\nu}^{m}$, $\sigma_{\nu}^{m}$, $W_{\nu,1}^{m}$ and $W_{\nu,2}^{m}$ are def\/ined with respect to either $h=a_{\nu}^{m}$ or $h=b_{\nu}^{m+1}$ in the previous section.

Let's f\/irst consider $m$ odd. In correspondence with the boundary conditions we require
\begin{gather}\label{boundcond}
\Ec_{\nu}^{m}\left(0,k^2\right)  =\frac{d \Ec_{\nu}^{m}\left(z,k^2\right)}{dz}\bigg|_{z=K}=0, \\ \label{boundcond2}
\Es_{\nu}^{m+1}\left(0,k^2\right)  =\Es_{\nu}^{m+1}\left(K,k^2\right)=0.
\end{gather}
The requirement at $z=K$ in (\ref{boundcond}) gives
\begin{gather*}
  \ifrac12\zeta_{*}W_{\nu,1}^{m}\left(\zeta_{*},k^2\right)+ \big(\zeta_{*}^2-\left(\sigma_{\nu}^{m}\right)^2\big)\frac{ d W_{\nu,1}^{m}\left(\zeta,k^2\right)}{d\zeta}\bigg|_{\zeta=\zeta_{*}} \\
 \qquad{}
+\eta_{\nu,c}^{m}\left( \ifrac12\zeta_{*}W_{\nu,2}^{m}\left(\zeta_{*},k^2\right)+ \big(\zeta_{*}^2-\left(\sigma_{\nu}^{m}\right)^2\big)\frac{d W_{\nu,2}^{m}\left(\zeta,k^2\right)}{d\zeta}\bigg|_{\zeta=\zeta_{*}}\right)=0,
\end{gather*}
thus from (\ref{errors2}), (\ref{errors1}) and (\ref{LSoln1}), and since as $\kappa\to\infty$ (see \cite[\S~12.9]{NIST:DLMF})
\begin{gather*}
U\left(-\ifrac{1}{2}\kappa\widetilde{\sigma}^2,\sqrt{2\kappa}\zeta_{*}\right)\sim e^{-\frac{\kappa\zeta_{*}^2}{2}}\left(\sqrt{2\kappa}\zeta_{*}\right)^{\ifrac{1}{2}\left(\kappa\widetilde{\sigma}^2-1\right)},\\  \nonumber
U'\left(-\ifrac{1}{2}\kappa\widetilde{\sigma}^2,\sqrt{2\kappa}\zeta_{*}\right)\sim-\sqrt{\frac{\kappa}{2}}\zeta_{*}U\left(-\ifrac{1}{2}\kappa\widetilde{\sigma}^2,\sqrt{2\kappa}\zeta_{*}\right), \\ \nonumber
\overline{U}\left(-\ifrac{1}{2}\kappa\widetilde{\sigma}^2,\sqrt{2\kappa}\zeta_{*}\right)\sim\sqrt{\frac{2}{\pi}}
\Gamma\left(\ifrac{1}{2}\left(1+\kappa\widetilde{\sigma}^2\right)\right)e^{\frac{\kappa\zeta_{*}^2}{2}}\left(\sqrt{2\kappa}
\zeta_{*}\right)^{-\ifrac{1}{2}\left(\kappa\widetilde{\sigma}^2+1\right)}, \\
\overline{U}'\left(-\ifrac{1}{2}\kappa\widetilde{\sigma}^2,\sqrt{2\kappa}\zeta_{*}\right)\sim\sqrt{\frac{\kappa}{2}}
\zeta_{*}\overline{U}\left(-\ifrac{1}{2}\kappa\widetilde{\sigma}^2,\sqrt{2\kappa}\zeta_{*}\right),
\end{gather*}
necessarily $\eta_{\nu,c}^{m}$ is exponentially small in this limit. Note that here we have assumed that $\kappa\widetilde{\sigma}^2=\mathcal{O}(1)$ as $\kappa\to\infty$,
and we show later, in~(\ref{LEigen5}), that this assumption is justif\/ied.

Similarly the requirement at $z=K$ in (\ref{boundcond2}) gives
\begin{gather*}
W_{\nu,1}^{m}\left(\zeta_{*},k^2\right)+\eta_{\nu,s}^{m+1}W_{\nu,2}^{m}\left(\zeta_{*},k^2\right)=0,
\end{gather*}
then necessarily $\eta_{\nu,s}^{m+1}$ is exponentially small as~$\kappa\to\infty$.

Hence for both these cases by considering the requirements at $z=0$ we have the condition
\begin{gather*}
U\left(-\ifrac{1}{2}\kappa\left(\widetilde{\sigma}_{\nu}^{m}\right)^2,0\right) +\mathcal{O}\left(\kappa^{-1}\right) =\frac{2^{\left(\kappa\left(\widetilde{\sigma}_{\nu}^{m}\right)^2-1\right)/4}\sqrt{\pi}}
{\Gamma\left(\frac{1}{4}\left(3-\kappa\left(\widetilde{\sigma}_{\nu}^{m}\right)^2\right)\right)}+\mathcal{O}\left(\kappa^{-1}\right)=0.
\end{gather*}
and to satisfy this we require
\begin{gather*}
\ifrac{1}{2}\kappa\left(\widetilde{\sigma}_{\nu}^{m}\right)^2=j+\ifrac{1}{2}+\mathcal{O}\left(\kappa^{-1}\right),
\end{gather*}
where $j$ is an odd integer.

Let's now consider $m$ even. In correspondence with the boundary conditions we require
\begin{gather*}
\frac{d \Ec_{\nu}^{m}\left(z,k^2\right)}{dz}\bigg|_{z=0}=\frac{d \Ec_{\nu}^{m}\left(z,k^2\right)}{dz}\bigg|_{z=K}=0, \\
\frac{d  \Es_{\nu}^{m+1}\left(z,k^2\right)}{dz}\bigg|_{z=0} =\Es_{\nu}^{m+1}\left(K,k^2\right)=0.
\end{gather*}
By similar reason to the $m$ odd case, when considering the boundary condition at $z=K$ both~$\eta_{\nu,c}^{m}$ and~$\eta_{\nu,s}^{m+1}$ are necessarily exponentially small as~$\kappa\to\infty$. Hence by using the boundary condition at $z=0$ we have the condition
\begin{gather*}
U'\left(-\ifrac{1}{2}\kappa\left(\widetilde{\sigma}_{\nu}^{m}\right)^2,0\right)
+\mathcal{O}\left(\kappa^{-1}\right)=-\frac{2^{\left(\kappa\left(\widetilde{\sigma}_{\nu}^{m}\right)^2+1\right)/4}
\sqrt{\pi}}{\Gamma\left(\frac{1}{4}\left(1-\kappa\left(\widetilde{\sigma}_{\nu}^{m}\right)^2\right)\right)}+\mathcal{O}\left(\kappa^{-1}\right)=0.
\end{gather*}
To satisfy this we require that
\begin{gather*}
\ifrac{1}{2}\kappa\left(\widetilde{\sigma}_{\nu}^{m}\right)^2=j+\ifrac{1}{2}+\mathcal{O}\left(\kappa^{-1}\right),
\end{gather*}
where $j$ is an even integer.

As $\kappa\to\infty$, the zeros of $U\big({-}\frac{1}{2}\kappa\left(\widetilde{\sigma}_{\nu}^{m}\right)^2,\sqrt{2\kappa}\zeta\big)$ tend to the zeros of $D_{j}\big(\sqrt{2\kappa}\zeta\big)$, the parabolic cylinder functions with exactly $j$ zeros in it's oscillatory interval (see \cite[\S~12.11]{NIST:DLMF}). Thus in correspondence with the number of zeros of the Lam\'{e} functions and those of $D_{j}\big(\sqrt{2\kappa}\zeta\big)$, we deduce that $j=m$ and as such from~(\ref{LEigen4}) and~(\ref{LEigen44}) we have that
\begin{gather}\label{LEigen5}
\left(\sigma_{\nu}^{m}\right)^2=\frac{2m+1}{\kappa}+\mathcal{O}\left(\kappa^{-2}\right).
\end{gather}
We deduce from (\ref{LEigen4}) and (\ref{LEigen1}) that
\begin{gather}\label{Eigenproof}
a_{\nu}^{m}=(2m+1)\kappa+\mathcal{O}(1),\qquad
b_{\nu}^{m+1}=(2m+1)\kappa+\mathcal{O}(1), \qquad \kappa\to\infty.
\end{gather}

\section[Approximations in terms of parabolic cylinder $D$ functions]{Approximations in terms of parabolic cylinder $\boldsymbol{D}$ functions}\label{CaseB}

The Lam\'{e} functions decay exponentially on either side of the oscillatory interval in~$(-K,K)$. Our approximant in Section~\ref{CaseA}, $U\left(-\frac{1}{2}\kappa\widetilde{\sigma}^2,\sqrt{2\kappa}\zeta\right)$, displays the appropriate exponentially decreasing behaviour when $\zeta$ is large and positive, but when the variable is large and negative it becomes exponentially large. If our argument $-\frac{1}{2}\kappa\widetilde{\sigma}^2$ had been exactly a negative half-integer, which it is not, then the approximant would have exhibited this wanted exponentially decaying behaviour when the variable is both large and positive and large and negative.

With respect to (\ref{LEigen5}) we def\/ine
\begin{gather}\label{LEigen55}
\left({\omega}_{\nu}^{m}\right)^2=\left({\sigma}_{\nu}^{m}\right)^2-\frac{2m+1}{\kappa}=\mathcal{O}\left(\kappa^{-2}\right)  \qquad (\kappa\to\infty),
\end{gather}
thus it makes sense to split up (\ref{LDiff2.5}) so that
\begin{gather}\label{LDiff4}
\frac{d^{2}W_{\nu}^{m}\left(\zeta,k^2\right)}{d\zeta^{2}} =\left(\kappa^2\left(\zeta^2-\frac{2m+1}{\kappa}\right)+\widehat{\psi}_{\nu}^{m}\left(\zeta,k^2\right)\right)W_{\nu}^{m}\left(\zeta,k^2\right),
\end{gather}
where
\begin{gather*}
\widehat{\psi}_{\nu}^{m}\left(\zeta,k^2\right)=\psi_{\nu}^{m}\left(\zeta,k^2\right)-\kappa^2\left({\omega}_{\nu}^{m}\right)^2.
\end{gather*}
From (\ref{LEigen55}) and since $s_{\nu}^{m}$, $\sigma_{\nu}^{m}$ and the variables $x$ and $\zeta$ are all bounded as $\kappa\to\infty$, we have for $\zeta\in[-\zeta_{*},\zeta_{*}]$
\begin{gather*}
\widehat{\psi}_{\nu}^{m}\left(\zeta,k^2\right)=\mathcal{O}(1),
\end{gather*}
in this limit uniformly. Now our approximant will have the desired property of decaying exponentially on either side of the oscillatory interval for large positive and large negative $\zeta$. We obtain the solutions for (\ref{LDiff4})
\begin{gather}
W_{\nu,1}^{m}\left(\zeta,k^2\right) =D_{m}\left(\sqrt{2\kappa}\zeta\right)+\epsilon_{\nu,1}^{m}\left(\zeta,k^2\right),\nonumber\\
 W_{\nu,2}^{m}\left(\zeta,k^2\right) =\overline{D}_{m}\left(\sqrt{2\kappa}\zeta\right)+\epsilon_{\nu,2}^{m}\left(\zeta,k^2\right),\label{LSoln2}
\end{gather}
valid when $\zeta\in[0,\zeta_{*}]$, where we denoted for convenience
\begin{gather*}
\overline{D}_{m}(t)=\overline{U}\left(-m-\ifrac{1}{2},t\right).
\end{gather*}
We have that
\begin{gather*}
\widehat{\psi}_{\nu}^{m}\left(0,k^2\right)=-\kappa^2\left({\omega}_{\nu}^{m}\right)^2+\psi_{\nu}^{m}\left(0,k^2\right)
=-\kappa^2\left({\sigma}_{\nu}^{m}\right)^2+(2m+1)\kappa+\frac{\left({s}_{\nu}^{m}\right)^4-\left({\sigma}_{\nu}^{m}\right)^4}{2\left({\sigma}_{\nu}^{m}\right)^2\left({s}_{\nu}^{m}\right)^4}
\end{gather*}
In Part \ref{parttwo}, Section~\ref{eigenvalues}, we will show that
\begin{gather*}
\left({\sigma}_{\nu}^{m}\right)^2 =\frac{2m+1}{\kappa}-\frac{1+k^2}{8\kappa^2}+\mathcal{O}\left(\kappa^{-3}\right).
\end{gather*}
We combine this with (\ref{LLEigen}) and obtain
\begin{gather}\label{psi01}
\widehat{\psi}_{\nu}^{m}\left(0,k^2\right) =\mathcal{O}\left(\kappa^{-1}\right) \qquad (\kappa\to\infty).
\end{gather}

This time in the error analysis we choose it so that
\begin{gather*}
\Omega(z)=1+z,
\end{gather*}
and, hence, take as the variational operator
\begin{gather*}
\mathcal{V}_{a,b}\left(\widetilde{\psi}_{\nu}^{m}\right) = \int_{a}^{b}\frac{\big|\widetilde{\psi}_{\nu}^{m}\left(t,k^2\right)\big|}{1+\sqrt{2\kappa}t}dt.
\end{gather*}
The bounds for the errors are
\begin{gather*}
\left|\epsilon_{\nu,1}^{m}\left(\zeta,k^2\right)\right|  \leq \frac{\textbf{M}_{m}\left(\sqrt{2\kappa}\zeta\right)}{\textbf{E}_{m}\left(\sqrt{2\kappa}\zeta\right)}
\left[\exp\left(\ifrac{1}{2}\sqrt{\pi/\kappa}  l_{m,1}\mathcal{V}_{\zeta,\zeta_{*}}\big(\widetilde{\psi}_{\nu}^{m}\big)\right)-1\right],\\
\left|\epsilon_{\nu,2}^{m}\left(\zeta,k^2\right)\right|  \leq \textbf{E}_{m}\big(\sqrt{2\kappa}\zeta
\big)\textbf{M}_{m}\big(\sqrt{2\kappa}\zeta
\big)\left[\exp\left(\ifrac{1}{2}\sqrt{\pi/\kappa}  l_{m,1}\mathcal{V}_{0,\zeta}\big(\widetilde{\psi}_{\nu}^{m}\big)\right)-1\right],
\end{gather*}
where
\begin{gather*}
l_{m,1}=\sup_{z\in(0,\infty)}\left\{\frac{\Omega(z)\textbf{M}^2_{m}(z)}{\Gamma(m+1)}\right\}.
\end{gather*}
We def\/ine for convenience
\begin{gather*} \nonumber
\textbf{E}\left(-m-\ifrac{1}{2},z\right)=\textbf{E}_{m}(z), \qquad \textbf{M}\left(-m-\ifrac{1}{2},z\right)=\textbf{M}_{m}(z), \qquad \textbf{N}\left(-m-\ifrac{1}{2},z\right)=\textbf{N}_{m}(z).
\end{gather*}
We have the bounds
\begin{gather*}
\mathcal{V}_{a,b}\big(\widetilde{\psi}_{\nu}^{m}\big) = \int_{a}^{b}\frac{\big|\widetilde{\psi}_{\nu}^{m}\big(t,k^2\big)\big|}{1+\sqrt{2\kappa}t}dt\leq  \int_{a}^{b}\frac{\big|\widetilde{\psi}_{\nu}^{m}\big(0,k^2\big)\big|}{1+\sqrt{2\kappa}t}dt+\int_{a}^{b}
\frac{\big|\widetilde{\psi}_{\nu}^{m}\big(t,k^2\big)-\widetilde{\psi}_{\nu}^{m}\big(0,k^2\big)\big|}{\sqrt{2\kappa}t}dt,
\end{gather*}
thus from (\ref{psi01}) we have
\begin{gather*}
\mathcal{V}_{0,\zeta}\big(\widehat{\psi}_{\nu}^{m}\big)=\mathcal{O}\big(\kappa^{-1/2}\big) \qquad \textrm{and} \qquad  \mathcal{V}_{\zeta,\zeta_{*}}\big(\widehat{\psi}_{\nu}^{m}\big)=\mathcal{O}\big(\kappa^{-1/2}\big) \qquad \textrm{as} \quad \kappa\to\infty.
\end{gather*}
From (\ref{Olvref}) clearly $l_{m,1}$ is a bounded constant as $\kappa\to\infty$.
Hence we have
\begin{gather}\label{errorsL}
\epsilon_{\nu,1}^{m}\left(\zeta,k^2\right) = \frac{\textbf{M}_{m}\left(\sqrt{2\kappa}\zeta \right)}{\textbf{E}_{m}\left(\sqrt{2\kappa}\zeta\right)}\mathcal{O}\left(\kappa^{-1}\right),
 \qquad  \frac{\epsilon_{\nu,2}^{m}\left(\zeta,k^2\right)}{ \textbf{E}_{m}\left(\sqrt{2\kappa}\zeta\right)} =\textbf{M}_{m}\big(\sqrt{2\kappa}\zeta \big)\mathcal{O}\left(\kappa^{-1}\right).
\end{gather}
Applying similar analysis to the above, again from Theorem~1 in \cite[\S~6]{Olver1975}, we obtain
\begin{gather*}
\frac{\partial \epsilon_{1}\left(\zeta,k^2\right)}{\partial\zeta} = \frac{\textbf{N}_{m}\left(\sqrt{2\kappa}\zeta\right)}{\textbf{E}_{m}\left(\sqrt{2\kappa}\zeta\right)}\mathcal{O}\big(\kappa^{-1/2}\big),
\qquad \frac{{\partial \epsilon_{2}\left(\zeta,k^2\right)}/{\partial\zeta}}{\textbf{E}_{m}\left(\sqrt{2\kappa}\zeta\right)} = \textbf{N}_{m}\big(\sqrt{2\kappa}\zeta\big)\mathcal{O}\big(\kappa^{-1/2}\big).
\end{gather*}

Thus from (\ref{LSoln2}), (\ref{LLiouv}) and (\ref{LTran}) we obtain
\begin{gather*}
\Ec_{\nu}^{m}\left(z,k^2\right) =C_{\nu}^{m}\left(\frac{\zeta^2-\left(\sigma_{\nu}^{m}\right)^2}{x^2-\left(s_{\nu}^{m}\right)^2}\right)^{1/4}\\
\hphantom{\Ec_{\nu}^{m}\left(z,k^2\right) =}{}
\times \left(D_{m}\big(\sqrt{2\kappa}\zeta\big)+\epsilon_{\nu,1}^{m}\left(\zeta,k^2\right)
 +\eta_{\nu,c}^{m} \left(\overline{D}_{m}\big(\sqrt{2\kappa}\zeta\big)+\epsilon_{\nu,2}^{m}\left(\zeta,k^2\right)\right)\right),\\
\Es_{\nu}^{m+1}\left(z,k^2\right) = S_{\nu}^{m+1}\left(\frac{\zeta^2-\left(\sigma_{\nu}^{m}\right)^2}{x^2-\left(s_{\nu}^{m}\right)^2}\right)^{1/4}\\
\hphantom{\Es_{\nu}^{m+1}\left(z,k^2\right) =}{}
\times \left(D_{m}\big(\sqrt{2\kappa}\zeta\big)+\epsilon_{\nu,1}^{m}\left(\zeta,k^2\right)
 +\eta_{\nu,s}^{m+1} \left(\overline{D}_{m}\big(\sqrt{2\kappa}\zeta\big)+\epsilon_{\nu,2}^{m}\left(\zeta,k^2\right)\right)\right),
\end{gather*}
valid as $\kappa\to\infty$ for $z\in[0,K]$, where $\epsilon_{\nu,1}^{m}\left(\zeta,k^2\right)$ and $\epsilon_{\nu,2}^{m}\left(\zeta,k^2\right)$ are given in~(\ref{errorsL}), and $\eta_{\nu,c}^{m}$ and $\eta_{\nu,s}^{m+1}$ are constants to be determined. Since we have the boundary conditions
\begin{gather*}
\frac{d \Ec_{\nu}^{m}\left(z,k^2\right)}{dz}\bigg|_{z=K}=0, \qquad \Es_{\nu}^{m+1}\left(K,k^2\right)=0,
\end{gather*}
we obtain that necessarily both $\eta_{\nu,c}^{m}$ and $\eta_{\nu,s}^{m+1}$ are $\mathcal{O}\big(e^{-\kappa\zeta_{*}^2}{\kappa}^{m+1/2}\big)$ as $\kappa\to\infty$. Only when~$\zeta$ nears the endpoint the contribution from the $\overline{D}_{m}\big(\sqrt{2\kappa}\zeta\big)+\epsilon_{\nu,2}^{m}\left(\zeta,k^2\right)$ term is comparable to $D_{m}\big(\sqrt{2\kappa}\zeta\big)+\epsilon_{\nu,1}^{m}\left(\zeta,k^2\right)$.
These functions will be made unique by their normalisation.

\subsection*{A second term in the approximation in terms of parabolic cylinder $\boldsymbol{D}$ functions}

In~\cite[\S~11]{Olver1975} uniform asymptotic expansions are included. The two term version is of the form
\begin{gather*}
W_{\nu,1}^{m}\left(\zeta,k^2\right)\sim A_0D_{m}\big(\sqrt{2\kappa}\zeta\big)+\frac{B_0(\zeta,\kappa)}{\kappa^2}\frac{d}{d\zeta}D_{m}\big(\sqrt{2\kappa}\zeta\big),
\end{gather*}
in which we have taken $A_0=1$ and for $B_0(\zeta,\kappa)$ we have
\begin{gather*}
B_0(\zeta,\kappa)=\int_{\tilde\sigma}^\zeta \frac{\widehat{\psi}_{\nu}^{m}\left(t,k^2\right)}{2\sqrt{\left(\zeta^2-{\tilde\sigma}^2\right)\left(t^2-{\tilde\sigma}^2\right)}}  dt,
\end{gather*}
where $\tilde\sigma=\sqrt{(2m+1)/\kappa}$. This coef\/f\/icient depends on $\kappa$ and its dominant part is
\begin{gather}\label{BB0}
B_0(\zeta)=\lim_{\kappa\to\infty}B_0(\zeta,\kappa)=\frac1{2\zeta}\int_0^\zeta\frac{\lim\limits_{\kappa\to\infty}\widehat{\psi}_{\nu}^{m}\left(t,k^2\right)}{t} dt.
\end{gather}
To simplify the integrand we take the f\/inal equation in~(\ref{LInts}) and let $\kappa\to\infty$. The result is
\begin{gather*}
\ifrac12\zeta^2=\int_0^x\frac{t  dt}{\sqrt{\left(1-t^2\right)\left(1-k^2t^2\right)}}=\frac1{k}\ln\left(\frac{1+k}{\sqrt{1-k^2x^2}+k\sqrt{1-x^2}}\right).
\end{gather*}
This relation can be inverted
\begin{gather*}
x^2=k^{-1}\sinh\left(k\zeta^2\right)-\left(1+k^{-2}\right)\sinh^2\left(k\zeta^2/2\right).
\end{gather*}
We use this in
\begin{gather*}
\lim_{\kappa\to\infty}\widehat{\psi}_{\nu}^{m}\left(\zeta,k^2\right)
=\frac3{4\zeta^2}+\frac{\zeta^2}{4}\left(\frac{k^2\left(x^2+x^4\right)+x^2-3}{x^4}\right)+\frac{1+k^2}8,
\end{gather*}
and are able to evaluate the integral in (\ref{BB0}) and obtain
\begin{gather*}
32\zeta B_0(\zeta)=\left(k^2+1\right)\ln\left(\ifrac14\zeta^2C(\zeta,k)\right)-\frac{3\left(k^2-1\right)^2}{2C(\zeta,k)}
+3k\coth\left(k\zeta^2/2\right)+2k^2\zeta^2-\frac{6}{\zeta^2},
\end{gather*}
where $C(\zeta,k)=2k\coth\left(k\zeta^2/2\right)-k^2-1$. Thus our two term approximation is
\begin{gather*}
\Ec_{\nu}^{m}\left(z,k^2\right)\sim C_{\nu}^{m}\left(\frac{\zeta^2-\left(\sigma_{\nu}^{m}\right)^2}{x^2-\left(s_{\nu}^{m}\right)^2}\right)^{1/4}\left(D_{m}\big(\sqrt{2\kappa}\zeta\big)
+\frac{B_0(\zeta)}{\kappa^2}\frac{d}{d\zeta}D_{m}\big(\sqrt{2\kappa}\zeta\big)\right),
\end{gather*}
as $\kappa\to\infty$ and similarly for $\Es_{\nu}^{m+1}\left(z,k^2\right)$.

\subsection*{Normalisation constant}
We now want an approximation for $C_{\nu}^{m}$ and $S_{\nu}^{m+1}$.
Due to the complicated nature of the mapping between $x$ and $\zeta$, it is not simple to
express one in terms of the other.
With respect to~(\ref{LNorm}) we consider the integral
\begin{gather*}
\left(C_{\nu}^{m}\right)^2\int_{-K}^{K} \dn(z,k) \sqrt{\frac{\zeta^2-\left(\sigma_{\nu}^{m}\right)^2}{x^2-\left(s_{\nu}^{m}\right)^2}}D^2_{m}
\big(\sqrt{2\kappa}\zeta \big) dz.
\end{gather*}
Using only this f\/irst term in an expansion for the solution, we can only obtain coef\/f\/icients for up to~$1/h$ terms. Any further terms in an expansion for the solution would contribute to further terms in the normalisation constant expansion, which we will not consider here. Since the f\/irst terms in their respective function uniform approximations are the same for large~$\kappa$, it will be the same for both~$C_{\nu}^{m}$ and~$S_{\nu}^{m+1}$.

First we perform the transformation $x=\sn(z,k)$ to obtain
\begin{gather*}
\left(C_{\nu}^{m}\right)^2\int_{-1}^{1} \sqrt{\frac{\zeta^2-\left(\sigma_{\nu}^{m}\right)^2}{\big(x^2-\left(s_{\nu}^{m}\right)^2\big)\left(1-x^2\right)}}D^2_{m}
\big(\sqrt{2\kappa}\zeta \big) dx,
\end{gather*}
then another transformation to the $\zeta$ variable to obtain
\begin{gather*}
\left(C_{\nu}^{m}\right)^2\int_{-\zeta_{*}}^{\zeta^{*}} \frac{\zeta^2-\left(\sigma_{\nu}^{m}\right)^2}{x^2-\left(s_{\nu}^{m}\right)^2}\sqrt{1-k^2x^2}D^2_{m}
\big(\sqrt{2\kappa}\zeta \big) d\zeta.
\end{gather*}
As the oscillatory behaviour of $D_{m}\big(\sqrt{2\kappa}\zeta\big)$ happens in a shrinking region of the origin as $\kappa\to\infty$, we seek to approximate the integral around this point to get an approximation for~$C_{\nu}^{m}$. Thus we consider an expansion of the form
\begin{gather}\label{Term1}
x=\sum_{k=1}^{\infty}c_{k}\zeta^{k},
\end{gather}
as $\zeta\to 0$, and substituting this into the relation
\begin{gather}\label{Term2}
\frac{dx}{d\zeta}=\sqrt{\frac{\big(\zeta^2-\left(\sigma_{\nu}^{m}\right)^2\big)\left(1-x^2\right)\left(1-k^2x^2\right)}{x^2-\left(s_{\nu}^{m}\right)^2}}.
\end{gather}
From this we can determine, by matching~(\ref{Term1}) and~(\ref{Term2}), the $c_{k}$ terms, and obtain
\begin{gather*}
x=\frac{\sigma_{\nu}^{m }}{s_{\nu}^{m}}\zeta+\left(\frac{(\sigma_{\nu}^{m })^4-(s_{\nu}^{m })^4}{6(s_{\nu}^{m })^5\sigma_{\nu}^{m }}-\frac{\left(1+k^2\right)(\sigma_{\nu}^{m })^3}{6(s_{\nu}^{m })^3}\right)\zeta^3+\mathcal{O}\left(\zeta^5\right),
\end{gather*}
as $\zeta\to 0$. From this we obtain
\begin{gather*}
\frac{\zeta^2-\left(\sigma_{\nu}^{m}\right)^2}{x^2-\left(s_{\nu}^{m}\right)^2}\sqrt{1-k^2x^2}
=\frac{\left(\sigma_{\nu}^{m}\right)^2}{\left(s_{\nu}^{m}\right)^2}
+\frac{2\left(\sigma_{\nu}^{m}\right)^4-2\left(s_{\nu}^{m}\right)^4-k^2\left(s_{\nu}^{m}\right)^2
\left(\sigma_{\nu}^{m}\right)^4}{2\left(s_{\nu}^{m}\right)^6}\zeta^2+\mathcal{O}\left(\zeta^4\right),
\end{gather*}
as $\zeta\to 0$. With respect to this, we consider as a f\/irst approximation for $C_{\nu}^{m}$ the integral
\begin{gather*}
\left(C_{\nu}^{m}\right)^2\int_{-\zeta_{*}}^{\zeta_{*}}\left(\frac{\left(\sigma_{\nu}^{m}\right)^2}{\left(s_{\nu}^{m}\right)^2}
+\frac{\left(2\sigma_{\nu}^{m}\right)^4-2\left(s_{\nu}^{m}\right)^4
-k^2\left(s_{\nu}^{m}\right)^2\left(\sigma_{\nu}^{m}\right)^4}{2\left(s_{\nu}^{m}\right)^6}\zeta^2\right)D^2_{m}\big(\sqrt{2\kappa}\zeta\big)d\zeta.
\end{gather*}
and letting $t=\sqrt{2\kappa}\zeta$, we have the approximation
\begin{gather*}
\frac{\left(C_{\nu}^{m}\right)^2}{\sqrt{2\kappa}}\int_{-\infty}^{\infty}\left(1+\frac{\left(1+k^2\right)(2m+1)+t^2-k^2t^2}{8\kappa}\right)D^2_{m}(t)dt
\sim\frac{\left(C_{\nu}^{m}\right)^2m!}{\sqrt{\kappa/\pi}}\left(1+\frac{2m+1}{4\kappa}\right).
\end{gather*}
In correspondence with (\ref{Lsigns}) and that for large $z$,
\begin{gather*}
D_{m}(z)>0, \qquad D_{m}'(z)<0 \qquad \textrm{(see \cite[\S~12.9(i)]{NIST:DLMF}},
\end{gather*}
we deduce from (\ref{LNorm}) that as $\kappa\to\infty$
\begin{gather*}\left.
\begin{array}{@{}l@{}}
C_{m} \\
S_{m+1} \\
\end{array} \right\}
\sim\frac{\left(\pi \kappa\right)^{{1}/{4}}}{\sqrt{2m!}}\left(1-\frac{2m+1}{8\kappa}\right).
\end{gather*}

\section{A special case: Mathieu functions}\label{Mathieu}

We obtain rigorous uniform results for Mathieu's equation, using the limiting arguments given in~(\ref{Msoln}) and~(\ref{MEigen1}).
Thus (reader can check the details, we will just summarise) we have the uniform approximations as $h\to\infty$
\begin{gather*}
\ce_{m}(h,z) = C_{m}\left(\frac{\zeta^2-\sigma_{m}^{2}}{x^2-s_{m}^{2}}\right)^{1/4}\\
\hphantom{\ce_{m}(h,z) =}{}
\times\left(D_{m}\big(2\sqrt{h}\zeta\big)+\epsilon_{m,1}(\zeta)
 +\eta_{m}^{c} \big(\overline{D}_{m}\big(\sqrt{2}h\zeta\big)+\epsilon_{m,2}(\zeta)\big)\right),\\
\se_{m+1}(h,z) = S_{m+1}\left(\frac{\zeta^2-\left(\sigma_{\nu}^{m}\right)^2}{x^2-\left(s_{\nu}^{m}\right)^2}\right)^{1/4}\\
\hphantom{\se_{m+1}(h,z) =}{}
\times
\left(D_{m}\big(2\sqrt{h}\zeta\big)+\epsilon_{m,1}(\zeta)
 +\eta_{m+1}^{s} \big(\overline{D}_{m}\big(\sqrt{2}h\zeta\big)+\epsilon_{m,2}(\zeta)\big)\right),
\end{gather*}
valid for $\zeta\in[0,\zeta_{*}]$, i.e.,  $z\in[0,\frac{\pi}{2}]$, as $h\to\infty$, where
\begin{gather*}
\epsilon_{m,1}(\zeta) =\frac{\textbf{M}_{m}\left(\sqrt{2}h\zeta \right)}{\textbf{E}_{m}\left(\sqrt{2}h\zeta\right)}\mathcal{O}\left(h^{-1}\right), \qquad \epsilon_{m,2}(\zeta)=\textbf{E}_{m}\big(\sqrt{2}h\zeta\big)\textbf{M}_{m}\big(\sqrt{2}h\zeta \big)\mathcal{O}\left(h^{-1}\right)
\end{gather*}
both $\eta_{m}^{c}$ and $\eta_{m+1}^{s}$ are $\mathcal{O}\big(e^{-2h\zeta_{*}^2}h^{m+1/2}\big)$ as $h\to\infty$,
the relationship between $z$ and $x$ is def\/ined by $x=\cos{z}$ and by evaluating the integrals (\ref{LInts}), the relationship between $x$ and $\zeta$ is def\/ined~by
\begin{alignat*}{3}
& \int_{x}^{-s_{m}}\sqrt{
\frac{t^{2}-s_{m}^2}{(1-t^{2})\left(1-k^2t^2\right)}}dt  =\int_{\zeta}^{-\sigma_{m}}\sqrt{\tau^{2}-\sigma_{m}^2}d\tau, \qquad && -1<x \leq -s_{m},&\\
 & \int_{-s_{m}}^{x}\sqrt
{\frac{s_{m}^2-t^{2}}{(1-t^{2})\left(1-k^2t^2\right)}}dt  = \int_{-\sigma_{m}}^{\zeta}\sqrt{\sigma_{m}^2-\tau^{2}} d\tau, \qquad && -s_{m}\leq x \leq s_{m},&\\
& \int_{s_{m}}^{x}\sqrt
{\frac{t^{2}-s_{m}^2}{(1-t^{2})\left(1-k^2t^2\right)}} dt  = \int_{\sigma_{m}}^{\zeta}\sqrt{\tau^{2}-\sigma_{m}^2}d\tau, \qquad && s_{m} \leq x < 1,
\end{alignat*}
where
\begin{gather*}
s^2_m=\frac{\lambda_{m}+2h^2}{4h^2}, \qquad \sigma_{m}^{2}=
s^2_{m} \, {}_2F_1\left({\frac12,\frac12\atop 2};s_{m}^2\right),
\end{gather*}
and $\lambda_{m}$ corresponds to either $a_{m}$ or $b_{m+1}$ depending on the solution we are considering, and
\begin{gather*}
\left.
\begin{array}{@{}l@{}}
C_{m} \\
S_{m+1}
\end{array}\right\}
\sim\left(\frac{\pi h}{2m!^2}\right)^{1/4}\left(1-\frac{2m+1}{16h}\right).
\end{gather*}
We have that
\begin{gather*}
s_{m}^2=\frac{m+\frac{1}{2}}{h}+\mathcal{O}\left(h^{-2}\right) \qquad \textrm{as} \quad h\to\infty,
\end{gather*}
thus from (\ref{LEigen1}) and (\ref{MEigen1}) we obtain
\begin{gather*}
a_{m}=-2h^2+(4m+2)h+\mathcal{O}(1),\qquad
b_{m+1}=-2h^2+(4m+2)h+\mathcal{O}(1),
\end{gather*}
as $h\to\infty$. One should note that this result about the eigenvalues was given formally in \cite[\S~2.33, Satz~5 \& Satz~6]{MS1954}.

Two term approximations are given by
\begin{gather*} \nonumber
\ce_{m}(h,z) \sim C_{m}\left( D_{m}\big(2\sqrt{h}\zeta\big)+\frac{B_0(\zeta)}{h^2}\frac{d}{d\zeta}D_{m}\big(2\sqrt{h}\zeta\big)\right),\\
\se_{m+1}(h,z) \sim S_{m+1}\left( D_{m}\big(2\sqrt{h}\zeta\big)+\frac{B_0(\zeta)}{h^2}\frac{d}{d\zeta}D_{m}\big(2\sqrt{h}\zeta\big)\right),
\end{gather*}
where for $B_0(\zeta)$ we have the relation
\begin{gather*}
256   B_0(\zeta)=\frac{3\zeta}{4-\zeta^2}-\frac{2}{\zeta}\ln\left(1-\frac{\zeta^2}{4}\right).
\end{gather*}

\part{Uniform asymptotic expansions}\label{parttwo}

In Part \ref{partone}, uniform asymptotic approximations were given for the Lam\'{e} and Mathieu functions when a parameter became large. A third term in an asymptotic expansion could not be computed due to the complicated nature of the transformation of the independent variable. In Section~\ref{expansion} of this part we employ a simpler transformation than in Part~\ref{partone} such that we can construct formal asymptotic expansions for the Lam\'{e} functions and their corresponding eigenvalues in the forms
\begin{gather*}
w_{\nu}^{m}(t)\sim D_{m}(t)\sum_{s=0}^{\infty}\frac{A_{s}(t)}{\kappa^s}+D_{m}'(t)\sum_{s=0}^{\infty}\frac{B_{s}(t)}{\kappa^{s}}, \qquad h\sim(2m+1)\kappa+\sum_{s=0}^{\infty}\frac{\mu_{s}}{\kappa^{s}},
\end{gather*}
where $t=\sqrt{2\kappa}\sn{(z,k)}$, the coef\/f\/icients $A_{s}(t)$ and $B_{s}(t)$ are just  polynomials,
and the $\mu_{s}$ are constants in terms of~$m$ and~$k$. The function expansions will clearly only make sense when $t=\mathcal{O}(1)$ as $\kappa\to\infty$. In Section~\ref{eigenvalues}, we analyse the asymptotic expansions for the eigenvalues and give an order estimate for the error corresponding to the truncated eigenvalue expansion. Then in Section~\ref{expansionerror} we use these results to give rigorous and realistic error bounds for the function expansions upon truncation. In Sections~\ref{identify2} and~\ref{summary2} we identify our expansions with the Lam\'{e} functions and then summarise our results. Finally in Section~\ref{specialmathieu2} we give analogous results for the Mathieu functions and their corresponding eigenvalues as a~special limiting case of those given for the Lam\'{e} functions and their corresponding eigenvalues.

\section{Uniform asymptotic expansions of the Lam\'{e} functions}\label{expansion}
In this section we construct formal asymptotic expansions for solutions of Lam\'{e}'s equation and corresponding eigenvalues. We do this by performing a simpler transformation than employed in Part~\ref{partone}. The oscillatory behaviour of the Lam\'{e} functions happens in a shrinking neighbourhood of the origin as $\kappa\to\infty$, and it can be shown that around the origin $\zeta$ behaves approximately like~$\sn(z,k)$. Thus the variable in the parabolic cylinder function around this point behaves approximately like~$\sqrt{2\kappa}\sn(z,k)$. This motivates the next simpler transformation.

\subsection*{$\boldsymbol{t}$-plane}
Letting $t=\sqrt{2\kappa}\sn(z,k)$ in (\ref{LDiff1}) we have
\begin{gather*}
\frac{d^2w}{dt^2}+\frac{2h-t^2\kappa}{4\kappa}w-\frac{1}{2\kappa} \left(\left(\left(1+k^2\right)t^2-\frac{k^2t^4}{2\kappa}\right)\frac{d^2}{dt^2}+t\left(1+k^2-\frac{k^2t^2}{\kappa}\right)\frac{d}{dt}\right) w=0,
\end{gather*}
where $z\in(-K,K)$ corresponds to $t\in\big({-}\sqrt{2\kappa},\sqrt{2\kappa}\big)$. Supposing in accordance with~(\ref{Eigenproof}) that
\begin{gather}\label{LLEigen2}
\frac{h}{2\kappa}=m+\ifrac{1}{2}+\sum_{s=1}^{n}\frac{\mu_{s}}{\kappa^{s}}+\frac{\widetilde{\mu}_{n+1}}{\kappa^{n+1}},
\end{gather}
where $n$ is a positive integer and $\widetilde{\mu}_{n}$ can be re-expanded in a sensible manner.
We can then write Lam\'{e}'s equation in the form
\begin{gather}\nonumber
 \frac{d^2w}{dt^2}+\left(m+\frac{1}{2}-\frac{t^2}{4}\right) w+\frac{1}{\kappa}\left\{\left(-\frac{1}{2}\left(1+k^2\right)t^2+\frac{k^2t^4}{4\kappa}\right)\frac{d^2}{dt^2}\right.
\\ \label{Ldiff2}
\left.\qquad{} +\frac{t}{2}\left(-1-k^2+\frac{k^2t^2}{\kappa}\right)
\frac{d}{dt}+\sum_{s=0}^{n}\frac{\mu_{s+1}}{\kappa^{s}}+\frac{\widetilde{\mu}_{n+1}}{\kappa^{n+1}}\right\} w=0.
\end{gather}
This equation is split in such a way that constructing a formal asymptotic expansion in terms of parabolic cylinder functions $D_{m}(t)$ in the form
\begin{gather}\label{soln}
w_{\nu}^{m}\left(t,k^2\right)=D_{m}(t)\sum_{s=0}^{\infty}\frac{A_{s}(t)}{\kappa^s}+D_{m}'(t)\sum_{s=0}^{\infty}\frac{B_{s}(t)}{\kappa^{s}}
\end{gather}
seems sensible. However one should observe that this splitting only makes sense when $t=\mathcal{O}(1)$ as $h\to\infty$, hence this formal expansion is only sensible for this range of~$t$. We denote this range as~$(-t_{*},t_{*})$. One should note that whilst this appears to be a new ansatz, there are similar expansions given in the literature for the Mathieu functions. These expansions are given in terms of~$D_{m-4j}(t)$ for $j\in \Z$ instead of in terms of~$D_{m}(t)$ and~$D_{m}'(t)$, but using the recurrence relations for the parabolic cylinder functions it is obvious these expansions are equivalent up to normalisation. An expansion in the form we have given allows us to dif\/ferentiate easily and hence seems the most natural in this case, also allowing us to perform rigorous error analysis.

We seek solutions which are either even or odd respective to the parity of $m$. Since $D_{m}(t)$ is either even or odd respective to when $m$ is either even or odd, we deduce that $A_{s}(t)$ and $B_{s}(t)$ must be even and odd respectively.

Substituting (\ref{soln}) into (\ref{Ldiff2}) and equating powers of $\kappa$, we have the recurrence relations for~$A_{s}(t)$ and~$B_{s}(t)$
\begin{gather}\nonumber
 2A_{s}'(t)+B_{s}''(t)- \frac{t+tk^2}{2}\left(2tA_{s-1}'(t)+A_{s-1}(t)+tB_{s-1}''(t)+B_{s-1}'(t)+\tilde{t}(m)tB_{s-1}(t)\right)\\ \nonumber
\qquad{}+\frac{k^2t^3}{4}\left(2tA_{s-2}'(t)+2A_{s-2}(t)+tB_{s-2}''(t)+2B_{s-2}'(t)+\tilde{t}(m)tB_{s-2}(t)\right)\\
 \qquad{}+\sum_{j=1}^{s}\mu_{j}B_{s-j}(t)=0,\label{rec1}
\\
A_{s}''(t)+2\tilde{t}(m)B_{s}'(t)+\frac{t}{2}B_{s}(t)+\sum_{j=1}^{s}\mu_{j}A_{s-j}(t)
-\frac{t+tk^2}{2}\bigl(tA_{s-1}''(t)+A_{s-1}'(t)+\frac{t^2}2B_{s-1}(t)\nonumber
\\
\qquad{} +\tilde{t}(m)\left(tA_{s-1}(t)+2tB_{s-1}'(t)+B_{s-1}(t)\right)\bigr)
 +\frac{k^2t^3}{4}\bigl(tA_{s-2}''(t)+2A_{s-2}'(t)+\frac{t^2}2B_{s-2}(t)\nonumber
 \\
\qquad{}+\tilde{t}(m)\left(tA_{s-2}(t)+2tB_{s-2}'(t)+2B_{s-2}(t)\right)\bigr) =0,\label{rec2}
\end{gather}
where $\tilde{t}(m)=\frac14t^2-m-\frac12$.
Neither of these relations separately determine solutions for $A_{s}(t)$ or $B_{s}(t)$ from previous coef\/f\/icients, thus we dif\/ferentiate (\ref{rec1}) to obtain an expression for $A''_{s}(t)$ and substitute it into~(\ref{rec2}); this gives the third order inhomogeneous dif\/ferential equation for~$B_{s}(t)$
\begin{gather}\label{diffB}
B_{s}'''(t)-\left(t^2-4m-2\right)B_{s}'(t)-tB_{s}(t)=b_{s}(t),
\end{gather}
where
\begin{gather}\nonumber
b_{s}(t)= \sum_{j=1}^{s}\mu_{j}\left(2A_{s-j}(t)-B_{s-j}'(t)\right)\\
 \nonumber
\hphantom{b_{s}(t)=}{} +\frac{1+k^2}{2}\left[3tA_{s-1}'(t)+A_{s-1}(t)+t^2B_{s-1}'''(t)+3tB_{s-1}''(t)+B_{s-1}'(t) \right.\\ \nonumber
 \left.\hphantom{b_{s}(t)=}{} -\ifrac12t^3B_{s-1}(t)-\tilde{t}(m)t^2\left(2A_{s-1}(t)+3B_{s-1}'(t)\right)\right]\\
 \nonumber
\hphantom{b_{s}(t)=}{} -\frac{k^2t^2}{2}\left[3tA_{s-2}'(t)+3A_{s-2}(t)+\ifrac12t^2B_{s-2}'''(t)+3tB_{s-2}''(t)+3B_{s-2}'(t)\right.\\
 \left.\hphantom{b_{s}(t)=}{}
 -\ifrac14t^3B_{s-2}(t)-\ifrac12t^2\tilde{t}(m)\left(2A_{s-2}(t)+3B_{s-2}'(t)\right)\right].\label{littleb}
\end{gather}
Once $B_{s}(t)$ is determined, we can use (\ref{rec1}) to determine~$A_{s}(t)$.
There will be freedom in choosing the integration constants in the~$A_{s}(t)$ terms, with identif\/ication of our solutions made unique by their normalisation.

\subsection*{General coef\/f\/icients $\boldsymbol{B_{s}(t)}$ and $\boldsymbol{A_{s}(t)}$}

Using variation of parameters we obtain the general solution for $B_{s}(t)$
\begin{gather*}
B_{s}(t)=b^{1}_{s}(t) D^2_{m}(t) + b^{2}_{s}(t)  \overline{D}^2_{m}(t) + b^{3}_{s}(t) D_{m}(t)\overline{D}_{m}(t),
\end{gather*}
where $\overline{D}_{m}(t)=\overline{U}\left(-m-\frac12,t\right)$ (for $\overline{U}$ see formula~(12.2.21) in \cite{NIST:DLMF}), and
\begin{gather*}
b^{1}_{s}(t)  =\int \frac{b_{s}(t)\mathcal{W}\{\overline{D}^2,D\overline{D}\}}{\mathcal{W}\{D^2,\overline{D}^2,D\overline{D}\}}dt+c_{1},\qquad
b^{2}_{s}(t) =-\int \frac{b_{s}(t) \mathcal{W}\{D^2,D\overline{D}\}}{\mathcal{W}\{D^2,\overline{D}^2,D\overline{D}\}}dt+c_{2},\\
b^{3}_{s}(t) =\int \frac{b_{s}(t) \mathcal{W}\{D^2,\overline{D}^2\}}{\mathcal{W}\{D^2,\overline{D}^2,D\overline{D}\}}dt+c_{3}.
\end{gather*}
By expanding the Wronskians we derive the relations
\begin{alignat*}{3}
& \mathcal{W}\{D^2,D\overline{D}\} =\mathcal{W}\{D,\overline{D}\} D^2_{m}(t),\qquad && \mathcal{W}\{\overline{D}^2,D\overline{D}\}=-\mathcal{W}\{D,\overline{D}\} \overline{D}^2_{m}(t),& \\
& \mathcal{W}\{D^2,\overline{D}^2\} =2\mathcal{W}\{D,\overline{D}\} D_{m}(t)\overline{D}_{m}(t), \qquad&&
\mathcal{W}\{D^2,\overline{D}^2,D\overline{D}\}=-2 \mathcal{W}\{D,\overline{D}\}^3.&
\end{alignat*}
Then without loss of generality we can rewrite the indef\/inite integrals in $\left\{b^{i}_{s}(t)\right\}_{i=1}^{3}$ as def\/inite integrals from $0$ to $t$, and since $B_s(t)$ is supposed to be an odd function we have to
take $c_{1}=c_{2}=0$. Thus
\begin{gather}\label{Beqn}
B_{s}(t)=\frac{\pi}{4m!^2}\int_{0}^{t}b_{s}(\tau)\left(\overline{D}_{m}(\tau)D_{m}(t)-{D}_{m}(\tau)\overline{D}_{m}(t)\right)^2d\tau +c_{3} D_{m}(t)\overline{D}_{m}(t).
\end{gather}
Although we consider $t$ to be $\mathcal{O}(1)$ as $\kappa\to\infty$, we still want an expansion which exhibits the correct behaviour at the endpoints of the interval. Expanding the squared term in~(\ref{Beqn}) and splitting it into three separate integrals, large variable asymptotics for the parabolic cylinder functions (see \cite[\S~12.9]{NIST:DLMF}) tells us that the terms involving $D_{m}^2(t)$ or ${D}_{m}(t)\overline{D}_{m}(t)$ grow no faster than polynomials when~$t$ becomes large. Since
\begin{gather*}
\overline{D}_{m}(t)^2 \sim \frac{2m!^2}{\pi}e^{t^2/2}t^{-2m-2}, \qquad t\to\infty,
\end{gather*}
to ensure the boundedness of our formal expansion as~$t$ becomes large, $\mu_{s}$ is determined uniquely by the condition that
\begin{gather}\label{cond}
\int_{0}^{\infty} b_{s}(\tau){D}_{m}^2(\tau)\,d\tau=0.
\end{gather}
This condition will also ensure boundedness as $t\to-\infty$. Note that in the next section we will derive an alternative method to compute the~$\mu_{s}$ terms which we use to compute the coef\/f\/icients as it is simpler.

Consider f\/irst $s=0$. From~(\ref{littleb}) we see that $b_{0}(t)=0$, thus ensuring that~$B_{0}(t)$ is odd we obtain the general solution
\begin{gather*}
B_{0}(t)=c_{3}D_{m}(t)\overline{D}_{m}(t),
\end{gather*}
and then
\begin{gather*}
 A_{0}(t)=-\ifrac{1}{2}c_{3}\big(D_{m}'(t)\overline{D}_{m}(t)+D_{m}(t)\overline{D}_{m}'(t)\big)+c.
\end{gather*}
Rearranging, on the $\kappa^{0}$ level of the asymptotic expansion (\ref{soln}) we have
\begin{gather*}
\left(\tfrac{1}{2}c_{3}\mathcal{W}(D_{m},\overline{D}_{m})+c\right)D_{m}(t).
\end{gather*}
Thus having this term in $B_{0}(t)$ equates to $B_{0}(t)=0$ and an extra constant term in $A_{0}(t)$, and since we have freedom in the arbitrary constant terms in $A_{s}(t)$, we can take $c_{3}=0$
without loss of generality. For simplicity we take $A_{0}(t)=1$ and adopt the convention $A_{s}(t)=0$ for $s \geq 1$. With these choices we have
\begin{gather*}
b_{1}(t)=2\mu_{1}+\frac{1+k^2}4\left(2-t^2\left(t^2-4m-2\right)\right),
\end{gather*}
and expressing the $D_m(t)$ in (\ref{cond}) via \cite[formula~(12.7.2)]{NIST:DLMF} in terms of Hermite polynomials we can evaluate the integral in (\ref{cond}) and obtain
\begin{gather*}
\mu_{1}=-\ifrac{1}{8}\left(1+k^2\right)\left(1+2m+2m^2\right).
\end{gather*}
In the case $s=1$, it can be shown using integration by parts that for $\mu_{1}$ which satisf\/ies~(\ref{cond})
\begin{gather*}
 \frac{\pi}{4m!^2}\int_{0}^{t}b_{1}(\tau)\big(\overline{D}_{m}(\tau)D_{m}(t)-{D}_{m}(\tau)\overline{D}_{m}(t)\big)^2d\tau\\
\qquad{} =\frac{1+k^2}{16}\left(t^3-(2m+1)t+ (-1 )^m\frac{2m+1}{m!}\sqrt{\frac{\pi}{2}}{D}_{m}(t)\overline{D}_{m}(t)\right).
\end{gather*}
Thus clearly in this case, if the correct $c_{3}$ is chosen in~(\ref{Beqn}) then $B_{1}(t)$ is exactly an odd polynomial. Using similar observations as in the $s=0$ case, we also note that if this multiple of ${D}_{m}\overline{D}_{m}$ is included in the $B_{1}(t)$ term, the expansion can be rearranged so that this term is instead represented in the constant term of~$A_{1}(t)$. Taking $B_{1}(t)$ to be exactly an odd polynomial, we get that $A_{1}(t)$ is an even polynomial. If one would go through the details to compute the representation for $B_{1}(t)$ as a polynomial plus this multiple of ${D}_{m}\overline{D}_{m}$, one would see that it would appear for all~$s\geq1$ that if the previous~$A_{s}(t)$ and $B_{s}(t)$ terms are all polynomials, then the $A_{s}(t)$ and $B_{s}(t)$ terms can be represented as polynomials.

\subsection*{Polynomial coef\/f\/icients $\boldsymbol{B_{s}(t)}$ and $\boldsymbol{A_{s}(t)}$}

To obtain explicit expressions for $A_{s}(t)$ and $B_{s}(t)$ we try substituting in polynomial expansions with undetermined coef\/f\/icients.
Take $A_{0}(t)=1$ and $B_{0}(t)=0$, then for $s\geq 1$ we consider $A_{s}(t)$ and $B_{s}(t)$ in the form
\begin{gather*}
A_{s}(t)=\sum_{i=1}^{\infty}a_{s,i}t^{2i} \qquad \textrm{and} \qquad B_{s}(t)=\sum_{i=0}^{\infty}b_{s,i}t^{2i+1}.
\end{gather*}
Substituting these into (\ref{diffB}) and (\ref{rec1}) we obtain the recurrence relations for coef\/f\/icients
\begin{gather} \nonumber
(2i+3)(2i+2)(2i+1)b_{s,i+1}+(4m+2)(2i+1)b_{s,i}-2ib_{s,i-1}\\
\nonumber \qquad{}
+2\sum_{j=0}^{s-1}\mu_{s-j}\left(\left(i+\ifrac12\right)b_{j,i}-a_{j,i}\right)
 -\left(1+k^2\right)\Bigl[\left(3i+\ifrac12\right)a_{s-1,i}+\ifrac{1}{2}\left(2i+1\right)^3b_{s-1,i}\\
 \nonumber
 \qquad{}
 +\left(m+\ifrac{1}{2}\right)\left(a_{s-1,i-1}+\left(3i-\ifrac32\right)b_{s-1,i-1}\right)
  -\ifrac{1}{4}a_{s-1,i-2}-\left(\ifrac{3}{4}i-\ifrac{7}{8}\right)b_{s-1,i-2}\Bigr] \\
 \nonumber
 \qquad{} +k^2\Bigl[\left(3i-\ifrac32\right)a_{s-2,i-1}\!+2i\left(i^2\!-\ifrac14\right)b_{s-2,i-1} \!
 +(2m+1)\left(\ifrac{1}{4}a_{s-2,i-2}\!+\left(\ifrac{3}{4}i-\ifrac{9}{8}\right)b_{s-2,i-2}\right)\\
 \qquad{}
-\ifrac{1}{8}a_{s-2,i-3}-\left(\ifrac{3}{8}i-\ifrac{13}{16}\right)b_{s-2,i-3}\Bigr]=0,   \label{recc1}
\\
\nonumber
2(2i+2)a_{s,i+1}+(2i+3)(2i+2)b_{s,i+1}+\sum_{j=0}^{s-1}\mu_{s-j}b_{j,i}\\
\qquad{}-\ifrac{1}{2}\left(1+k^2\right)\left((4i+1)a_{s-1,i}+(2i+1)^2b_{s-1,i}
-\left(m+\ifrac{1}{2}\right)b_{s-1,i-1}+\ifrac{1}{4}b_{s-1,i-2}\right)\label{recc2} \\
\qquad{}+k^2\left(i\left(i-\ifrac12\right)b_{s-2, i-1}+\left(i-\ifrac12\right)a_{s-2,i-1}
-\ifrac18 (2m+1 )b_{s-2,i-2}+\ifrac{1}{16}b_{s-2,i-3}\right)=0.\nonumber
\end{gather}
From these recurrence relations it is observed that only a f\/inite number of the $a_{s,i}$ and $b_{s,i}$ are non-zero. The orders are displayed in Table~\ref{table:order}.

\begin{table}[h]\centering
\caption{The orders of the polynomials.}\label{table:order}
\vspace{1mm}

    \begin{tabular}{ | p{0.8cm} | p{2cm} | p{2cm}|}
\hline
$s$&$A_{s}(t)$&$B_{s}(t)$  \\
\hline
   even & $4s$&$4s-3$ \\
   odd & $4s-2$&$4s-1$   \\
    \hline
\end{tabular}
\end{table}

In this manner the coef\/f\/icients are determined recursively. We deduce by considering the dif\/ference of (\ref{recc1}) and (\ref{recc2}) in the case $i=0$ that
\begin{align}
\mu_{s}=(2m+1)b_{s,0}-2a_{s,1}.
\end{align}
In the case that $s$ is even then we start with $i=2s-1$ in  (\ref{recc1}) to determine $b_{s,2s-2}$. Then we take $i=2s-2$ and determine $b_{s,2s-3}$, and so on.
Once every power of $t$ is eliminated, we are left with a constant equation which we must make zero by specifying $\mu_{s}$; in this manner the eigenvalue terms are determined uniquely.
These~$\mu_{s}$ terms are the same as those specif\/ied by the condition~(\ref{cond}), since it is required for solutions $B_{s}(t)$ which grow no faster than polynomials.
Once the coef\/f\/icients in the polynomial expression for $B_{s}(t)$ are determined, and the corresponding~$\mu_{s}$, then the coef\/f\/icients in the polynomial expression for $A_{s}(t)$ can be determined from~(\ref{recc2}).
In the case that $s$ is odd then we have to start with $i=2s$ in~(\ref{recc1}) to determine~$b_{s,2s-1}$.

\subsection*{Returning to the $\boldsymbol{z}$-plane: odd solutions}
In \cite[\S~28.8]{NIST:DLMF} expansions for the odd Mathieu functions are given which exhibit the correct odd behaviour, which an expansion in the form~(\ref{soln}) would not when transformed back into the $z$-plane. We want something similar for the Lam\'{e} functions.

Once we transform our formal expansion back into the $z$-plane, over the whole real line they behave like the even Lam\'{e} functions $\Ec_{\nu}^{m}\left(z,k^2\right)$. We want to construct a similar expansions which when transformed back into the $z$-plane behave like the odd Lam\'{e} functions $\Es_{\nu}^{m+1}\left(z,k^2\right)$.  We need our expansion to have the property
\begin{gather*}
w_{\nu}^{m}\left(-K,k^2\right)=w_{\nu}^{m}\left(K,k^2\right)=0.
\end{gather*}
The Jacobi elliptic function $\cn(z,k)$ is even around the origin and odd around $z=-K$ and $z=K$. Hence again letting $t=\sqrt{2\kappa}\sn(z,k)$ we consider a formal expansion for a solution of~(\ref{Ldiff2}) of the form
\begin{gather*}
w_{\nu}^{m}\left(z,k^2\right)=\cn(z,k)\left(D_{m}(t)\sum_{s=0}^{\infty}\frac{P_{s}(t)}{\kappa^s}+D_{m}'(t)\sum_{s=1}^{\infty}\frac{Q_{s}(t)}{\kappa^{s}}\right).
\end{gather*}
This has the correct behaviour at $z=-K$ and $z=K$. Again we will require that~$P_{s}(t)$ is even and~$Q_{s}(t)$ is odd. By writing
\begin{gather*}
\cn\left(\arcsn\left(\frac{t}{\sqrt{2\kappa}},k\right),k\right)=\sqrt{1-\frac{t^2}{2\kappa}},
\end{gather*}
we can express the formal solution in the form
\begin{gather*}
w^{m}_{\nu}\left(z,k^2\right)
=\sqrt{1-\frac{t^2}{2\kappa}}\left(D_{m}(t)\sum_{s=0}^{\infty}\frac{P_{s}(t)}{\kappa^s}+D_{m}'(t)\sum_{s=1}^{\infty}\frac{Q_{s}(t)}{\kappa^{s}}\right).
\end{gather*}
Expanding this square root, we can rewrite this formal expansion as
\begin{gather*}
w_{\nu}^{m}\left(z,k^2\right)=D_{m}(t)\sum_{s=0}^{\infty}\frac{A_{s}(t)}{\kappa^s}+D_{m}'(t)\sum_{s=1}^{\infty}\frac{B_{s}(t)}{\kappa^{s}},
\end{gather*}
where the connection between $P_{s}(t)$ and $A_{s}(t)$, and $Q_{s}(t)$ and $B_{s}(t)$ is given by
\begin{gather*}
 A_{s}(t)=\sum_{j=0}^{s}\binom{\frac{1}{2}}{j}\left(-\frac{t^2}{2}\right)^{j}P_{s-j}(t),\qquad B_{s}(t)=\sum_{j=0}^{s}\binom{\frac{1}{2}}{j}\left(-\frac{t^2}{2}\right)^{j}Q_{s-j}(t),
\end{gather*}
where $\binom{a}{b}$ is the generalised binomial coef\/f\/icient. Thus we determine the $P_{s}(t)$ and $Q_{s}(t)$ terms by connection with the $A_{s}(t)$ and $B_{s}(t)$ derived previously. By connection with the previous case we get the same eigenvalue expansion to all orders for both~$a_{\nu}^{m}$ and~$b_{\nu}^{m+1}$.

\section{An interlude: error analysis for the eigenvalue expansions}\label{eigenvalues}
In this section we match our eigenvalue expansions with the eigenvalues $a_{\nu}^{m}$ and $b_{\nu}^{m+1}$, and give order estimates for the expansions on truncation. The theory of this subsection is included in~\cite{MS1954}, where it is employed for Mathieu's equation. For this we consider Sturm--Liouville theory; for a fuller treatment on Sturm--Liouville theory see \cite{Sturm08}. We have the dif\/ferential equation
\begin{gather*}
\frac{d}{dx}\left[p(x)\frac{dy}{dx}\right]+\left(\lambda w(x)-q(x)\right)y=0
\end{gather*}
and we consider $p(x),w(x) >0$, and $p(x)$, $p'(x)$, $q(x)$ and $w(x)$ to be continuous functions over a f\/inite real interval~$[a,b]$. When~$(a,b)$ is bounded and $p(x)$ does not vanish on $[a,b]$, this is a regular SL problem, otherwise it is singular. Consider the regular problem. We want to f\/ind special values of~$\lambda$ called eigenvalues for which there exists a non-trivial solution satisfying the separated boundary conditions
\begin{alignat}{3}
& \alpha_{1}y(a)+\alpha_{2}y'(a)=0, \qquad && \alpha_{1}^2+\alpha_{2}^2>0, & \nonumber \\
& \beta_{1}y(b)+\beta_{2}y'(b) =0, \qquad && \beta_{1}^2+\beta_{2}^2>0. & \label{bcs}
\end{alignat}
The regular SL theory states that these eigenvalues are real and can be ordered such that
\begin{gather*}
\lambda_{0}<\lambda_{1}<\cdots <\lambda_{\nu}<\cdots \to\infty
\end{gather*}
and these eigenvalues $\lambda_{\nu}$ correspond to unique eigenfunctions $y_{\nu}(x)$ with exactly $\nu$ zeros in $(a,b)$. Normalised eigenfunctions form an orthonormal basis such that
\begin{gather*}
(y_{n},y_{m})=\int_{a}^{b}y_{n}(x)y_{m}(x)w(x)dx=\delta_{mn}
\end{gather*}
in the Hilbert space $L^{2}\left[[a,b],w(x)dx\right]$.

The following theorem is Theorem~1 in \cite[\S~1.52]{MS1954}
\begin{Theorem}
Consider $SL$ to be a self-adjoint differential operator defined on a subspace~$A$ of $L^{2}\left[[a,b],w(x)dx\right]$ containing all the twice differentiable functions which satisfy the boundary conditions~\eqref{bcs} and corresponding~${\lambda}_{\nu}$ such that
\begin{gather*}
({\rm SL}+{\lambda}_{\nu})y_{\nu}(x)=0.
\end{gather*}
Let $\widetilde{y}(x)\in A$ be such that $(\widetilde{y},\widetilde{y})=1$. Now define a remainder function $R(x)$ such that
\begin{gather*}
({\rm SL}+{\lambda})\widetilde{y}(x)=R(x)
\end{gather*}
for some constant ${\lambda}$. If $R(x)\in L^{2}\left[[a,b],w(x)dx\right]$ then
\begin{gather*}
\min_{k}|{\lambda}-{\lambda}_{k}|<\sqrt{(R,R)}.
\end{gather*}
\end{Theorem}

\begin{proof}
We have that
\begin{gather*}
(R,y_{k})=((SL+{\lambda})\widetilde{y},y_{k})=(\widetilde{y},SLy_{k})+{\lambda}(\widetilde{y},y_{k})=({\lambda}-{\lambda}_{k})(\widetilde{y},y_{k}),
\end{gather*}
and using this in
\begin{gather*}
1=(\widetilde{y},\widetilde{y})=\sum_{k}|(\widetilde{y},y_{k})|^2=\sum_{k}\frac{|(R,y_{k})|^2}{|{\lambda}-{\lambda}_{k}|^2}
<\frac{(R,R)}{\min_{k}|{\lambda}-{\lambda}_{k}|^2},
\end{gather*}
we have that
\begin{gather*}
\min_{k}|{\lambda}-{\lambda}_{k}|<\sqrt{(R,R)}.\tag*{\qed}
\end{gather*}
\renewcommand{\qed}{}
\end{proof}

\subsection*{Eigenvalues of Lam\'{e}'s equation}
We now apply this theory to Lam\'{e}'s equation to obtain order estimates upon truncation of asymptotic expansions of the eigenvalues~$a_{\nu}^{m}$ and~$b_{\nu}^{m+1}$.

Special eigenvalues $a_{\nu}^{m}$ correspond to non-trivial solutions satisfying the separated boundary conditions
\begin{gather*}
\frac{d}{dz}\Ec_{\nu}^{m}\left(z,k^2\right)\bigg|_{z=-K}=\frac{d}{dz}\Ec_{\nu}^{m}\left(z,k^2\right)\bigg|_{z=K}=0
\end{gather*}
and these eigenvalues are real and can be ordered such that
\begin{gather*}
a^{0}_{\nu}<a^{1}_{\nu}<\cdots <a^{m}_{\nu}<\cdots \to\infty.
\end{gather*}
The functions $\Ec_{\nu}^{m}\left(z,k^2\right)$ have exactly $m$ zeros in $(-K,K)$. The normalised eigenfunctions $w_{\nu}^{m}\left(z,k^2\right)$ form an orthonormal basis such that
\begin{gather*}
\left(w_{\nu}^{i},w_{\nu}^{j}\right)=\int_{-K}^{K}w_{\nu}^{i}\left(z,k^2\right)w_{\nu}^{j}\left(z,k^2\right)dz=\delta_{ij}
\end{gather*}
in the Hilbert space $L^{2}\left[[-K,K],dz\right]$.

Let $L_{\nu}$ be the operator
\begin{gather*}
L_{\nu}:=\frac{d^2}{dz^2}-\kappa^2\sn^2(z,k)
\end{gather*}
so that
\begin{gather*}
\left(L_{\nu}+{a}_{\nu}^{m}\right)w_{\nu}^{m}\left(z,k^2\right)=0.
\end{gather*}
Letting $t=\sqrt{2\kappa}\sn(z,k)$, we def\/ine the truncated expansions corresponding to $w_{\nu}^{m}\left(z,k^2\right)$ as
\begin{gather}\label{wmn}
w_{\nu,n}^{m}\left(z,k^2\right) =c_{\nu,n}^{m}\left(D_{m}(t)\sum_{s=0}^{n}\frac{A_{s}(t)}{\kappa^s}+D_{m}'(t)\sum_{s=0}^{n}\frac{B_{s}(t)}{\kappa^{s}}\right),
\end{gather}
where the $A_{s}(t)$ and $B_{s}(t)$ terms were derived previously, and $c_{\nu,n}^{m}$ is def\/ined to be a function of~$\kappa$ so that
\begin{gather}\label{cnorm}
\int_{-K}^{K}\left(w_{\nu,n}^{m}\left(z,k^2\right)\right)^2dz=1.
\end{gather}
Then we can write the derivative of $w_{\nu,n}^{m}\left(z,k^2\right)$ with respect to $z$ as
\begin{gather*}
\frac{dw_{\nu,n}^{m}\left(z,k^2\right)}{dz}=c_{\nu,n}^{m}\sqrt{2\kappa}\cn(z,k)\dn(z,k) \\
\hphantom{\frac{dw_{\nu,n}^{m}\left(z,k^2\right)}{dz}=}{}
\times\left(D_{m}(t)\sum_{s=0}^{n}
\frac{{A}'_{s}(t)+\left(\frac{t^2}{4}-m-\frac{1}{2}\right)B_{s}(t)}{\kappa^s}
 +D'_{m}(t)\sum_{s=0}^{n} \frac{A_{s}(t)+{B}'_{s}(t)}{\kappa^s}\right),
\end{gather*}
where the dash represents dif\/ferentiation with respect to $t$.
Thus clearly $w_{\nu,n}^{m}\left(z,k^2\right)\in A$ since $\cn(-K,k)=\cn(K,k)=0$. We also def\/ine the truncated eigenvalue expansions
\begin{gather*}
a_{\nu,n}^{m} =(2m+1)\kappa+2\sum_{s=0}^{n-1}\frac{\mu_{s+1}}{\kappa^{s}},
\end{gather*}
where the $\mu_{s}$ were derived previously and def\/ine $R_{\nu,n}^{m}\left(z,k^2\right)$ such that
\begin{gather*}
\left(L_{\nu}+{a}_{\nu,n}^{m}\right)w_{\nu,n}^{m}\left(z,k^2\right)=R_{\nu,n}^{m}\left(z,k^2\right).
\end{gather*}
We consider the operator $L_{\nu}$ acting on $D_{m}(t)$ and derive
\begin{gather}\nonumber
L_{\nu}\left(D_{m}(t)\right)= \left((2m+1)\left(\frac{1+k^2}{2}t^2-\kappa\right)
-\left(k^2(2m+1)+\kappa\left(1+k^2\right)\right)\frac{t^4}{4\kappa}+\frac{k^2t^6}{8\kappa}\right)D_{m}(t)\\
 \label{Lop}
\hphantom{L_{\nu}\left(D_{m}(t)\right)=}{}
+\left(-\left(1+k^2\right)t+\frac{k^2 t^3}{\kappa}\right)D_{m}'(t).
\end{gather}
Using the recurrence relations
\begin{gather}\label{recc}
tD_{m}(t)=D_{m+1}(t)+mD_{m-1}(t),\qquad
D'_{m}(t)=mD_{m-1}(t)-\ifrac{1}{2}tD_{m}(t),
\end{gather}
we can rewrite $w_{\nu,n}^{m}\left(z,k^2\right)$ given in~(\ref{wmn}) such that for $s\in \{0,\dots ,n \}$, $A_{s}(t)D_{m}(t)$ and $B_{s}(t)D_{m}'(t)$ are sums of varying orders of parabolic cylinder functions with constant coef\/f\/icients dependent only on~$m$ and~$k$; it then follows from~(\ref{Lop}) that with the solution rewritten in this form, we have
\begin{gather*}
R_{\nu,n}^{m}\left(z,k^2\right)=c_{\nu,n}^{m}\kappa^{-n}\left[ (\cdots )+\kappa^{-1} (\cdots )+\cdots\right],
\end{gather*}
where the terms inside the brackets are sums of varying orders of parabolic cylinder functions with constant coef\/f\/icients dependent only on $k$ and $m$. Note that this remainder term will be f\/inite sum and clearly $R_{\nu,n}^{m}\left(z,k^2\right)\in L^{2}\left[[-K,K],dz\right]$. As $w_{\nu,n}^{m}\left(z,k^2\right)\in A$ we have that
\begin{gather*}
\min_{k}\big|{a}_{\nu,n}^{m}-{a}_{\nu}^{k}\big|^2< \sqrt{\left(R_{\nu,n}^{m},R_{\nu,n}^{m}\right)}.
\end{gather*}
In Part~\ref{partone}, Section~\ref{interlude}, we proved that
\begin{gather*}
a_{\nu}^{m}=(2m+1)\kappa+\mathcal{O}(1) \qquad \textrm{as} \quad \kappa\to\infty,
\end{gather*}
hence for $\kappa$ large enough, necessarily
\begin{gather*}
\min_{k}\big|{a}_{\nu,n}^{m}-{a}_{\nu}^{k}\big|^2=\big|{a}_{\nu,n}^{m}-{a}_{\nu}^{m}\big|^2
\end{gather*}
and thus
\begin{gather*}
\big|{a}_{\nu,n}^{m}-{a}_{\nu}^{m}\big| < \sqrt{\left(R_{\nu,n}^{m},R_{\nu,n}^{m}\right)}.
\end{gather*}

We need an order estimate for ${\left(R_{\nu,n}^{m},R_{\nu,n}^{m}\right)}$. The integrals we must consider then are of the form
\begin{gather*}
I=\int_{-K}^{K}D_{i}\big(\sqrt{2\kappa}\sn(z,k)\big)D_{j}\big(\sqrt{2\kappa}\sn(z,k)\big)dz,
\end{gather*}
where $i,j\in\N_{0}$. Performing the substitution $t=\sqrt{2\kappa}\sn(z,k)$ we obtain
\begin{gather*}
I=\frac{1}{\sqrt{2\kappa}}\int_{-\sqrt{2\kappa}}^{\sqrt{2\kappa}}\frac{1}{\sqrt{1-\frac{t^2}{2\kappa}}\sqrt{1-\frac{k^2t^2}{2\kappa}}}D_{i}(t)D_{j}(t)dt.
\end{gather*}

Since when $t$ is large $D_{m}$ is exponentially small it follows that
\begin{gather*}
I=\mathcal{O}\big(\kappa^{-\frac{1}{2}}\big), \qquad \textrm{as} \quad \kappa\to\infty.
\end{gather*}

Considering the f\/irst term of the expansion for $w_{\nu,n}^{m}\left(z,k^2\right)$ in~(\ref{wmn}) we have
\begin{gather*}
\left(c_{\nu,n}^{m}\right)^2\int_{-K}^{K}D_{m}^2\big(\sqrt{2\kappa}\sn(z,k)\big) dz
 =\frac{\left(c_{\nu,n}^{m}\right)^2}{\sqrt{2\kappa}}\int_{-\sqrt{2\kappa}}^{\sqrt{2\kappa}}
\frac{1}{\sqrt{1-\frac{t^2}{2\kappa}}\sqrt{1-\frac{k^2t^2}{2\kappa}}}D_{m}^2(t) dt \\
\hphantom{\left(c_{\nu,n}^{m}\right)^2\int_{-K}^{K}D_{m}^2\big(\sqrt{2\kappa}\sn(z,k\big)) dz}{}
 \sim \frac{\left(c_{\nu,n}^{m}\right)^2}{\sqrt{2\kappa}}\int_{-\infty}^{\infty}D_{m}^2(t)\,dt
=\left(c_{\nu,n}^{m}\right)^2m!\sqrt{\frac{\pi}{\kappa}},
\end{gather*}
as $\kappa\to\infty$. Hence it follows from (\ref{cnorm}) that $c_{m,n}=\mathcal{O}(\kappa^{1/4})$ as $\kappa\to\infty$. Similar observations would give us that			
\begin{gather*}
\int_{-K}^{K}\left(R_{m,n}(z)\right)^2dz=\mathcal{O}\left(\kappa^{-2n}\right),
\end{gather*}
and so
\begin{gather*}
{a}_{\nu}^{m}-{a}_{\nu,n}^{m}= \mathcal{O}\left(\kappa^{-n}\right),
\end{gather*}
as $\kappa\to\infty$. The error analysis for $b_{\nu,n}^{m}$ would in the same manner give the order estimate
\begin{gather*}
{b}_{\nu}^{m+1}-{b}_{\nu,n}^{m+1}= \mathcal{O}\left(\kappa^{-n}\right) \qquad \textrm{as} \quad \kappa\to\infty.
\end{gather*}
Hence we have from (\ref{LLEigen2}) that
\begin{gather*}
a_{\nu}^{m}=(2m+1)\kappa+2\sum_{s=0}^{n-1}\frac{\mu_{s+1}}{\kappa^s}+\mathcal{O}\left(\kappa^{-n}\right).
\end{gather*}
In the same manner we would deduce that
\begin{gather*}
b_{\nu}^{m+1}=(2m+1)\kappa+2\sum_{s=0}^{n-1}\frac{\mu_{s+1}}{\kappa^s}+\mathcal{O}\left(\kappa^{-n}\right).
\end{gather*}
This result also follows from the dif\/ference between $a_{\nu}^{m}$ and $b_{\nu}^{m+1}$ being exponentially small as $\kappa\to\infty$ (see \cite[\S~28.8]{NIST:DLMF}).

\section[Error analysis for the uniform asymptotic expansions of the Lam\'{e} functions]{Error analysis for the uniform asymptotic expansions\\ of the Lam\'{e} functions}\label{expansionerror}

We now use the results of Section~\ref{eigenvalues} to obtain strict and realistic error bounds for the functions expansions derived in Section~\ref{expansion}.
Def\/ine the dif\/ferential operator
\begin{gather*}
L_{\nu} :=\frac{d^2}{dt^2}+m+\frac{1}{2}-\frac{t^2}{4}\\
\hphantom{L_{\nu} :=}{}
 +\frac{1}{2\kappa}\left(\left(-t^2\left(1+k^2\right)+\frac{k^2t^4}{2\kappa}\right)
\frac{d^2}{dt^2}-t\left(1+k^2-\frac{k^2t^2}{\kappa}\right)\frac{d}{dt}+h-\kappa (2m+1 )\right),
\end{gather*}
and consider $t\in(-t_{*},t_{*})$. We have the truncated expansion corresponding to an even solution
\begin{gather*}
w_{\nu,n}^{m}\left(t,k^2\right)=D_{m}(t)\sum_{s=0}^{n}\frac{A_{s}(t)}{\kappa^s}+D_{m}'(t)\sum_{s=1}^{n}\frac{B_{s}(t)}{\kappa^{s}},
\end{gather*}
such that we have the exact solution
\begin{gather}\label{tsoln}
w_{\nu}^{m}\left(t,k^2\right)=w_{\nu,n}^{m}\left(t,k^2\right)+ \epsilon_{\nu,n}^{m}\left(t,k^2\right).
\end{gather}
We def\/ine the remainder term $R_{\nu,n}^{m}\left(t,k^2\right)$ such that
\begin{gather*}
L_{\nu}\left(w_{\nu,n}^{m}\left(t,k^2\right)\right)=R_{\nu,n}^{m}\left(t,k^2\right)
\end{gather*}
and split the eigenvalue such that
\begin{gather*}
\frac{h}{2\kappa}=m+\ifrac{1}{2}+\sum_{s=1}^{n}\frac{\mu_{s}}{\kappa^s}+\frac{\widetilde{\mu}_{n+1}}{\kappa^{n+1}},
\end{gather*}
and note we just proved that $\widetilde{\mu}_{n+1}=\mathcal{O}(1)$ as $\kappa\to\infty$.
Since the coef\/f\/icients $A_s(t)$ and $B_s(t)$ satisfy~(\ref{rec1}) and~(\ref{rec2}) it follows that
\begin{gather*}
R_{\nu,n}^{m}\left(t,k^2\right)=\mathcal{O}\left(\kappa^{-n-1}\right),
\end{gather*}
as $\kappa\to\infty$. Applying $L_{\nu}$ to~(\ref{tsoln}) we obtain
\begin{gather*}
 \left(\epsilon_{\nu,n}^{m}\right)''+\left(m+\frac{1}{2}-\frac{t^2}{4}\right)\epsilon_{\nu,n}^{m}=
\frac{1}{1-\frac{t^2}{2\kappa}\left(1+k^2-\frac{k^2t^2}{2\kappa}\right)}\biggl[
\frac{t}{2\kappa}\left(1+k^2-\frac{k^2t^2}{\kappa}\right)\left(\epsilon_{\nu,n}^{m}\right)' \\
 \qquad{}  +\left(t^2\left(-1-k^2+\frac{k^2t^2}{2\kappa}\right)\left(2m+1-\frac{t^2}{2}\right)
-h+\kappa\left(2m+1\right)\right)\frac{\epsilon_{\nu,n}^{m}}{2\kappa}-R_{\nu,n}^{m}\biggr],
\end{gather*}
and denoting the right hand side of this equation $\Omega_{\nu,n}^{m}\left(t,k^2\right)$, by use of variation of parameters we have
\begin{gather*}
\epsilon_{\nu,n}^{m}\left(t,k^2\right)=\frac{\sqrt{\pi/2}}{m!}\int_{t}^{t_{*}}
\left[D_{m}(t)\overline{D}_{m}(\tau)-D_{m}(\tau)\overline{D}_{m}(t)\right]\Omega_{\nu,n}^{m}\left(\tau,k^2\right)d\tau.
\end{gather*}

In accordance with \cite[\S~6.10.2]{Olv97} we def\/ine $J(\tau)=1$, $H(\tau)=1+k^2-\frac{k^2\tau^2}{2\kappa}$ and
\begin{gather*}
K\left(t,\tau\right)=\frac{\sqrt{\pi/2}}{m!}\left(D_{m}(t)\overline{D}_{m}(\tau)-D_{m}(\tau)\overline{D}_{m}(t)\right),\qquad
\phi(\tau)=\frac{-R_{\nu,n}^{m}\left(\tau,k^2\right)}{1-\frac{\tau^2}{2\kappa}H(\tau)},\\
 \psi_{0}(\tau)=\frac{\kappa\left(2m+1\right)-h-\tau^2H(\tau)\left(2m+1-\frac{\tau^2}{2}\right)
}{2\kappa-\tau^2H(\tau)},\qquad \psi_{1}(\tau)=\frac{\tau\left(1+k^2-\frac{k^2\tau^2}{\kappa}\right)}{2\kappa-\tau^2H(\tau)},
\\
\Phi(t)=\int_{t}^{t_{*}}\left|\phi(\tau)\,d\tau\right|,\qquad
\Psi_{0}(t)=\int_{t}^{t_{*}} \left|\psi_0(\tau)\,d\tau\right|,\qquad
\Psi_{1}(t)=\int_{t}^{t_{*}} \left|\psi_1(\tau)\,d\tau\right|.
\end{gather*}
Since we consider $t\in(-t_{*},t_{*})$ where $t_{*}=\mathcal{O}(1)$ as $\kappa\to\infty$, the error analysis is much simpler than the analysis Olver uses in \cite{Olver1975} as we have
\begin{gather*}
|K\left(t,\tau\right)| \leq k_{0}, \qquad \textrm{and} \qquad |\partial K\left(t,\tau\right)/\partial t| \leq k_{1},
\end{gather*}
where $k_{0}$ and $k_{1}$ are $\mathcal{O}(1)$ as $\kappa\to\infty$. Thus again in accordance with \cite[\S~6.10.2]{Olv97} we def\/ine
\begin{gather}\label{psandqs}
P_{0}(t)=k_{0},\qquad
Q(\tau)=1,\qquad
P_{1}(t)= k_{1}
\end{gather}
(we do not def\/ine $P_{2}(t)$ as we do not need to bound $|\partial^2 K (t,\tau )/\partial t^2|$ to carry out our analysis),
and f\/inally the constants
\begin{gather} \label{kappas}
\widetilde{\kappa}=1,\qquad
\widetilde{\kappa}_{0}=k_{0},\qquad
\widetilde{\kappa}_{1}=k_{1}.
\end{gather}
Hence it follows from Theorem 10.1 in \cite[\S ~6.10.2]{Olv97} that
\begin{gather}\label{error}
\big|\epsilon_{\nu,n}^{m}\big(t,k^2\big)\big| \leq P_{0}(t) \widetilde{\kappa}   \Phi(t)\exp\big[\widetilde{\kappa}_{0} \Psi_{0}(t)+\widetilde{\kappa}_{1}  \Psi_{1}(t)\big].
\end{gather}
Since
\begin{gather*}
h-\kappa(2m+1)=\mathcal{O}(1)
\end{gather*}
for both $h=a_{\nu}^{m}$ and $h=b_{\nu}^{m+1}$ we obtain
\begin{gather} \label{psis}
\Phi(t)=\mathcal{O}\left(\kappa^{-n-1}\right),\qquad
\Psi_{0}(t)=\mathcal{O}\left(\kappa^{-1}\right),\qquad
\Psi_{1}(t)=\mathcal{O}\left(\kappa^{-1}\right),
\end{gather}
as $\kappa\to\infty$.
Then substituting expressions from (\ref{psis}), (\ref{kappas}) and the f\/irst of (\ref{psandqs}) into (\ref{error}) we have
\begin{gather*}
\epsilon_{\nu,n}^{m}\left(t,k^2\right)=\mathcal{O}\left(\kappa^{-n-1}\right) \qquad \textrm{as} \quad \kappa\to\infty, \quad \textrm{for} \quad t\in\left(-t_{*},t_{*}\right).
\end{gather*}

\section{Identif\/ication of solutions}\label{identify2}
Now we identify the solutions derived in Section~\ref{expansionerror} with the Lam\'{e} functions. We give the identif\/ication for $t\in(-t_{*},t_{*})$
\begin{gather*}
\Ec_{\nu}^{m}\left(z,k^2\right) =C_{\nu}^{m}\biggl(D_{m}(t)\sum_{s=0}^{n}\frac{A_{s}(t)}{\kappa^s}+D_{m}'(t)\sum_{s=0}^{n}\frac{B_{s}(t)}{\kappa^{s}} \nonumber\\ 
\hphantom{\Ec_{\nu}^{m}\left(z,k^2\right) =}{}
+\ifrac{1}{2}\left(\epsilon_{\nu,n}^{m}\left(t,k^2\right)+(-1)^{m}\epsilon_{\nu,n}^{m}\left(-t,k^2\right)\right)\biggr),\\ \nonumber
\Es_{\nu}^{m+1}\left(z,k^2\right)= S_{\nu}^{m+1}\sqrt{1-\frac{t^2}{2\kappa}}\biggl(D_{m}(t)
\sum_{s=0}^{n}\frac{P_{s}(t)}{\kappa^s}+D_{m}'(t)\sum_{s=0}^{n}\frac{Q_{s}(t)}{\kappa^{s}}
 \\
\hphantom{\Es_{\nu}^{m+1}\left(z,k^2\right)= }{} +\ifrac{1}{2}\left(\epsilon_{\nu,n}^{m}\left(t,k^2\right)+(-1)^{m}\epsilon_{\nu,n}^{m}\left(-t,k^2\right)\right)\biggr), 
\end{gather*}
where the errors are def\/ined according to $h=a_{\nu}^{m}$ or $h=b_{\nu}^{m+1}$ respectively.
We consider now just $C_{\nu}^{m}$ since we will obtain the same asymptotic expansion from $S_{\nu}^{m+1}$ by construction. To obtain an asymptotic expansion for these constants we consider with respect to~(\ref{LNorm}) the integral
\begin{gather*}
\left(C_{\nu}^{m}\right)^2\int_{-K}^{K}\dn(z,k)\left(D_{m}(t)
\sum_{s=0}^{\infty}\frac{A_{s}(t)}{\kappa^s}+D_{m}'(t)\sum_{s=0}^{\infty}\frac{B_{s}(t)}{\kappa^{s}}\right)^2dz.
\end{gather*}
In the integral we let $t=\sqrt{2\kappa}\sn(z,k)$ and obtain
\begin{gather*}
\frac{\left(C_{\nu}^{m}\right)^2}{\sqrt{2\kappa}}\int_{-\sqrt{2\kappa}}^{\sqrt{2\kappa}}
\frac{1}{\sqrt{1-\frac{t^2}{2\kappa}}}\left(D_{m}(t)\sum_{s=0}^{\infty}\frac{A_{s}(t)}{\kappa^s}+D_{m}'(t)
\sum_{s=0}^{\infty}\frac{B_{s}(t)}{\kappa^{s}}\right)^2dt.
\end{gather*}
Since the parabolic cylinder functions are exponentially small when the variable is large, we consider the integral from $-\infty$ to $\infty$ and express the integral in the form
\begin{gather*}
 \frac{\left(C_{\nu}^{m}\right)^2}{\sqrt{2\kappa}}\int_{-\infty}^{\infty}\sum_{s=0}^{\infty}\kappa^{-s}\sum_{j=0}^{s}\binom{-\frac{1}{2}}{j}
 \left(t^2/2\right)^{j}\sum_{i=0}^{s-j}
\biggl(A_{i}(t)A_{s-j-i}(t)D_{m}^2(t) \\
\qquad{}+2A_{i}(t)B_{s-j-i}(t)D_{m}(t)D_{m}'(t)+B_{i}(t)B_{s-j-i}(t)\left(D_{m}'(t)\right)^2\biggr) dt.
\end{gather*}
Then from (\ref{LNorm}) we have the formal asymptotic expansions for the normalisation constants
\begin{gather*}\left.
\begin{array}{@{}l@{}}
C_{\nu}^{m} \\
S_{\nu}^{m+1}
\end{array} \right\}
\sim\frac{(\pi \kappa)^{1/4}}{\sqrt{2m!}}\left(1+\sum_{s=1}^{\infty}\frac{\eta_{s}}{\kappa^s}\right)^{-1/2}.
\end{gather*}

To obtain analytic expressions for the $\eta_{s}$ terms we need expressions for integrals of the form
\begin{gather*}
p(m,n) = \int_{-\infty}^{\infty} t^nD^2_{m}(t)dt,
\qquad q(m,n)=\int_{-\infty}^{\infty} t^{n}\left(D_{m}'(t)\right)^2dt, \\
r(m,n) =\int_{-\infty}^{\infty} t^{n+1}D_{m}(t)D_{m}'(t)dt.
\end{gather*}
We have the identities $p(m,2n+1)=q(m,2n+1)=r(m,2n+1)=0$ for all $m$, $n$, and
\begin{gather*}
p(m,0)=\int_{-\infty}^{\infty} D^2_{m}(t)dt=m!\sqrt{2\pi} \qquad \mathrm{and} \qquad \int_{-\infty}^{\infty} D_{m}(t)D_{n}(t)dt=0 \qquad \mathrm{for} \quad m\neq n,
\end{gather*}
and using (\ref{recc}) we deduce the expression
$p(m,2)= (2m+1 )m!\sqrt{2\pi}$.
Using integration by parts we obtain the recurrence relation for~$p(m,n)$
\begin{gather*}
p(m,n)
=\frac{2 (n-1 ) (2m+1 )}{n} p(m,n-2)+\frac{ (n-3 ) (n-2 ) (n-1 )}{n}p(m,n-4).
\end{gather*}
We know $p (m,0 )$ and $p (m,2 )$ thus consequently can determine $p(m,n)$ for any $n$ recursively. Similarly using integration by parts we obtain expressions for the other integrals in terms of $p(m,n)$:
\begin{gather*}
q(m,n) =\frac{n+3}{4(n+1)} p(m,n+2)-\left(m+\tfrac{1}{2}\right) p(m,n),\nonumber\\
r(m,n) =\frac{2m+1}{n+2} p(m,n+2)-\frac{n+4}{2(n+2)(n+3)}p(m,n+4).
\end{gather*}

\section[Summary of results for the Lam\'{e} functions and their respective eigenvalues]{Summary of results for the Lam\'{e} functions\\ and their respective eigenvalues}\label{summary2}

Let $\kappa=k\sqrt{\nu(\nu+1)}$ and $t=\sqrt{2\kappa}\sn{(z,k)}$. Then for $m$ a non-negative integer and $z=\mathcal{O}\big(\kappa^{-1/2}\big)$ as $\kappa\to\infty$
\begin{gather*}
\Ec_{\nu}^{m}\left(z,k^2\right)= C_{\nu}^{m}\left(D_{m}(t)\sum_{s=0}^{n}\frac{A_{s}(t)}{\kappa^s}+D_{m}'(t)
\sum_{s=0}^{n}\frac{B_{s}(t)}{\kappa^{s}}+\mathcal{O}\left(\kappa^{-n-1}\right)\right),\\
\frac{\Es_{\nu}^{m+1}\left(z,k^2\right)}{\cn(z,k)}= S_{\nu}^{m+1}\left(D_{m}(t)
\sum_{s=0}^{n}\frac{P_{s}(t)}{\kappa^s}+D_{m}'(t)\sum_{s=0}^{n}\frac{Q_{s}(t)}{\kappa^{s}}+\mathcal{O}\left(\kappa^{-n-1}\right)\right),
\end{gather*}
where
\begin{gather*}\left.
\begin{array}{@{}l@{}}
C_{\nu}^{m} \\
S_{\nu}^{m+1}
\end{array}\right\}
\sim \frac{(\pi \kappa)^{1/4}}{\sqrt{2m!}}\left(1+\sum_{s=1}^{\infty}\frac{\eta_{s}}{\kappa^s}\right)^{-1/2}.
\end{gather*}
Both $A_{s}(t)$ and $P_{s}(t)$ are even polynomials, and both  $B_{s}(t)$ and $Q_{s}(t)$ are odd polynomials. These polynomials are found recursively and we give here the f\/irst two terms:
\begin{gather*}
A_{0}=1, \qquad  A_{1}=\frac{k^2+1}{32}t^2,\qquad B_{0}=0, \qquad B_{1}=\frac{k^2+1}{16}\left(t^3-(2m+1)t\right),\\
P_{0}=1, \qquad P_{1}=\frac{k^2+9}{32}t^2,\qquad Q_{0}=0, \qquad Q_{1}=B_{1},\qquad \eta_{1}=\frac{3-k^2}{16}(2m+1).
\end{gather*}

Correspondingly we have the eigenvalue expansion as $\kappa\to\infty$
\begin{gather*}\left.
 \begin{array}{@{}l@{}}
a_{\nu}^{m} \\
b_{\nu}^{m+1}
\end{array}\right\}
=(2m+1)\kappa+2\sum_{s=0}^{n-1}\frac{\mu_{s+1}}{\kappa^s}+\mathcal{O}\left(\kappa^{-n}\right), \qquad \textrm{as} \quad \kappa\to\infty.
\end{gather*}
where the order term will be dif\/ferent in both cases. The $\mu_{s}$ terms are constant coef\/f\/icients which depend on~$k$ and~$m$, found recursively with the eigenfunction expansions. We give here the f\/irst two terms:
\begin{gather*}
\mu_{1}=-\frac{k^2+1}{8}\left(1+2m+2m^2\right),\qquad
\mu_{2}=-\frac{2m+1}{32}\left(\left(k^2-1\right)^2\left(1+m+m^2\right)-2k^2\right).
\end{gather*}

These eigenvalue coef\/f\/icients match the formal results given in  \cite[\S~29.7]{NIST:DLMF}.

\section{A special case: Mathieu functions}\label{specialmathieu2}

\subsection*{Mathieu's equation}

Similarly to the Lam\'{e} case, if one considers the two term uniform approximation for the Mathieu functions you would observe that the oscillatory behaviour of the Mathieu functions happen in a
shrinking neighbourhood of the $z=\frac{\pi}{2}$ as $h\to\infty$. It can be shown that in a shrinking neighbourhood of this point, $\zeta$ behaves approximately like~$\cos(z)$.
Thus the variable in the parabolic cylinder function around this point behaves approximately like~$\sqrt{2}h\cos(z)$. This would motivate a much simpler transformation.

Letting $t=2\sqrt{h}\cos{z}$ in (\ref{MDiff1}) we present Mathieu's equation in the algebraic form
\begin{gather*}
\frac{d^2w}{dt^2}+\left(\frac{\lambda+2h^2}{4h}-\frac{t^2}{4}\right)w
-\frac{1}{4h}\left(t^2\frac{d^2}{dt^2}+t\frac{d}{dt}\right)w=0,
\end{gather*}
and since we proved that
\begin{gather*}
\frac{\lambda+2h^2}{4}-h\left(m+\ifrac{1}{2}\right)=\mathcal{O}(1), \qquad \textrm{as} \quad  h\to\infty
\end{gather*}
for the special eigenvalues $\lambda$ which emit the Mathieu functions we have Mathieu's equation in the form
 \begin{gather*}
\frac{d^2w}{dt^2}+\left(m+\frac{1}{2}-\frac{t^2}{4}\right)w+\frac{1}{h}
\left(-\frac{t^2}{4}\frac{d^2}{dt^2}-\frac{t}{4}\frac{d}{dt}+\frac{\lambda+2h^2}{4}-h\left(m+\frac{1}{2}\right)\right)w=0.
\end{gather*}
In a similar manner to the previous sections coef\/f\/icients in the expansions for the functions and eigenvalues can be computed using the same ansatz, although they would only make sense asymptotically in a shrinking neighbourhood of $z=\frac{\pi}{2}$. However these are also realised by considering the Mathieu functions as a special case of the Lam\'{e} functions.

\subsection*{Summary of results as a special case of Lam\'{e}'s equation}
We obtain rigorous uniform results for Mathieu's equation, using the limiting arguments given in~(\ref{Msoln}) and~(\ref{MEigen1}).
Thus (reader can check the details, we will just summarise) letting $t=2\sqrt{h}\cos{z}$, for $m\geq 0$ and $z=\frac{\pi}{2}+\mathcal{O}(h^{-1/2})$ we have as $h\to\infty$
\begin{gather}\nonumber
\ce_{m}(h,a_{m},z)= C_{m}\left(D_{m}(t)\sum_{s=0}^{n}\frac{A_{s}(t)}{h^s}+D_{m}'(t)\sum_{s=0}^{n}\frac{B_{s}(t)}{h^{s}}+\mathcal{O}\left(h^{-n-1}\right)\right),\\ \label{funcexp}
\frac{\se_{m+1}(h,b_{m+1},z)}{\sin{z}}= S_{m+1}\left(D_{m}(t)\sum_{s=0}^{n}\frac{P_{s}(t)}{h^s}+D_{m}'(t)\sum_{s=0}^{n}\frac{Q_{s}(t)}{h^{s}}+\mathcal{O}\left(h^{-n-1}\right)\right),
\end{gather}
where
\begin{gather*}\left.
\begin{array}{@{}l@{}}
C_{m} \\
S_{m+1}
\end{array}\right\}
\sim \left(\frac{\pi h}{2 (m! )^2}\right)^{1/4}\left(1+\sum_{s=1}^{\infty}\frac{\eta_{s}}{h^s}\right)^{-1/2}.
\end{gather*}
Both $A_{s}(t)$ and $P_{m+1,s}(t)$ are even polynomials, and both  $B_{s}(t)$ and $Q_{m+1,s}(t)$ are odd polynomials. These polynomials are found recursively and we give here the f\/irst few terms:
\begin{gather*}
A_{0}=1, \qquad A_{1}=\frac{t^2}{2^6},\qquad B_{0}=Q_0=0, \qquad B_{1}=Q_1=\frac{t^3}{2^{5}}- (1+2m )\frac{t}{2^{5}},\\ P_{0}=1, \qquad P_{1}=\frac{9t^2}{2^6},\\  A_{2}=\frac{t^8}{2^{13}}- (1+2m ) \frac{t^6}{2^{11}}
+\left(9+10m+10m^2\right)\frac{t^4}{2^{12}}+\left(5+6m-12m^2-8m^3\right)\frac{t^2}{2^{12}}, \\
 B_{2}=\frac{t^5}{2^{8}}
- (1+2m )\frac{5t^3}{2^{11}}-\left(11+20m+20m^2\right)\frac{t}{2^{11}}, \\
 P_{2}=\frac{t^8}{2^{13}}- (1+2m )\frac{t^6}{2^{11}}
+\left(113+10m+10m^2\right)\frac{t^4}{2^{12}}+\left(5+6m-12m^2-8m^3\right)\frac{t^2}{2^{12}}, \\
Q_{2}=\frac{t^5}{2^7}- (1+2m )\frac{13t^3}{2^{11}}-\left(11-20m-20m^2\right)\frac{t}{2^{11}},
\qquad \eta_{1}=\frac{6m+3}{32}.
\end{gather*}

Correspondingly we have the eigenvalue expansion as $h\to\infty$
\begin{gather}\label{eigenexp}\left.
 \begin{array}{@{}l@{}}
a_{m} \\
b_{m+1}
\end{array} \right\}
=-2h^2+4h\sum_{s=0}^{n}\frac{\mu_{s}}{h^{s}}+\mathcal{O}\left(h^{-n}\right),
\end{gather}
where the order term will be dif\/ferent in both cases. The $\mu_{s}$ terms are constant coef\/f\/icients which depend on~$m$, found recursively with the eigenfunction expansions. We give here the f\/irst few terms:
\begin{gather*}
\mu_{0}=m+\ifrac{1}{2}, \qquad \mu_{1} =-\ifrac{1}{16} \big(1+2m+2m^2\big), \qquad \mu_{2}=-\ifrac{1}{128}\left(1+3m+3m^2+2m^3\right).
\end{gather*}

These function and eigenvalue coef\/f\/icients match the formal results given in~\cite[\S~28.8]{NIST:DLMF} and~\cite[\S~2]{MS1954}, and various other papers discussed at the start of this paper. Note that similar results to~(\ref{funcexp}) and~(\ref{eigenexp}) were given in~\cite{Kurz1979}, but there the function expansions did not make sense in the interval they were stated to hold in, and the method used there to obtain the coef\/f\/icients in the expansions for the functions and eigenvalue was very cumbersome. Our methods were simple and enabled us to perform rigorous error analysis on these expansions.

\subsection*{Acknowledgements}

The authors thank the referees for very helpful comments and suggestions for improving the presentation.

\pdfbookmark[1]{References}{ref}
\LastPageEnding

\end{document}